\documentclass{article}

\usepackage{amsmath}
\usepackage{amsthm}
\usepackage{amsfonts}
\usepackage{amssymb}

\usepackage{graphicx}

\newtheorem{proposition}{Proposition}
\newtheorem{theorem}[proposition]{Theorem}
\newtheorem{lemma}[proposition]{Lemma}
\newtheorem{corollary}[proposition]{Corollary}

\newtheorem{definition}{Definition}

\theoremstyle{remark}

\newtheorem{observation}{Observation}

\DeclareMathOperator{\arctanh}{arctanh}
\title{Rotation Vectors for Homeomorphisms of Non-Positively Curved Manifolds\footnote{MSC 37E45 37A99}}
\author{Pablo Lessa\footnote{CMAT, Facultad de Ciencias, Universidad de la Rep\'ublica, Uruguay}}

\begin{document}
\maketitle
\begin{abstract}
Rotation vectors, as defined for homeomorphisms of the torus that are isotopic to the identity, are generalized to such homeomorphisms of any complete Riemannian manifold with non-positive sectional curvature.  These generalized rotation vectors are shown to exist for almost every orbit of such a dynamical system with respect to any invariant measure with compact support.  The concept is then extended to flows and, as an application, it is shown how non-null rotation vectors can be used to construct a measurable semi-conjugacy between a given flow and the geodesic flow of a manifold.
\end{abstract}

\tableofcontents

\section{Introduction}
Since Poincar\'e introduced rotation numbers to study the dynamics of homeomorphisms of the circle there have been several types of `rotation objects' defined with the purpose of capturing the way the trajectories of a dynamical system `wind around' a manifold as time progresses (see \cite{Misiurewicz:2007} for an overview).

Perhaps the most direct generalization of rotation numbers has been the concept of rotation vectors as defined for homeomorphisms isotopic to the identity on the two-dimensional torus.  Although the situation is more complicated than for homeomorphisms of the circle (e.g. there may exist points without a well defined rotation vector, and different points may have different vectors associated to them) the concept has been successful at least in the following two ways:
\begin{enumerate}
\item All periodic orbits have a well defined rotation vector.  And, more generally, so do almost all points with respect to any invariant measure.
\item Simple hypotheses on the set of rotation vectors have strong dynamical consequences.  For example, the existence of three non-collinear rotation vectors implies that there are infinitely many periodic points and that topological entropy is positive (see \cite{MR1101087} and \cite{MR958891}).  Also, in some cases where there is a single rotation vector one can show that there exists a semi-conjugacy to a rigid rotation (see \cite{MR2501297}).
\end{enumerate}

On general manifolds, asymptotic cycles for flows (introduced in \cite{MR0088720}) and homological rotation vectors for homeomorphisms isotopic to the identity (e.g. see \cite{MR1325916}) both give information about how orbits `wind around' homology and have several applications.

Even on a compact hyperbolic surface, one can construct a flow possessing a periodic orbit which is homologically trivial (and hence defines a null asymptotic cycle) but homotopically non-trivial.  Furthermore, there are certain Lagrangian dynamical systems for which the critical value of the Abelian covering (whose group of covering transformations is the first homology group of the manifold over the integers) is strictly greater than that of the universal covering (see \cite[Section 2-7]{MR1720372}).  This suggests that between these two energy levels the trajectories may be exhibiting a behavior which is homologically trivial but homotopically non-trivial.

With these two examples in mind, it is natural to ask whether there exists a useful concept of rotation vector that is applicable to manifolds other than tori and which measures homotopical information of trajectories.

A first step in this direction was made in \cite{MR1093216} where `asymptotic homotopy cycles' are defined.  However, as the authors noted, the concept lacks a general existence theorem.

More recently, motivated by comparison theorems between different geodesic flows on the same Riemannian manifold, P.Boyland introduced rotation measures (see \cite{MR1802657}), which arise from associating to each trajectory of a flow on a fiber bundle (e.g. a Lagrangian flow) a geodesic of the base manifold (which is assumed to be compact and hyperbolic).  For an invariant measure supported on a homotopically non-trivial periodic orbit the associated rotation measure is supported on the minimizing geodesic of the same free homotopy class.  Hence rotation measures succeed in capturing homotopical information.

In this work, we introduce rotation vectors for homeomorphisms which are isotopic to the identity on any complete Riemannian manifold of non-positive sectional curvature.  In the case of the two-dimensional torus, our definition specializes to the usual one.  Furthermore, we prove that rotation vectors have the same existence properties on all manifolds under consideration (i.e. all periodic orbits, and almost all points with respect to any invariant measure with compact support, have rotation vectors).  This is our main result, the proof of which occupies the first seven sections of the paper.

In section \ref{periodic} we show that, for periodic orbits in free homotopy classes of positive length (e.g. all non-trivial homotopy classes on a compact hyperbolic manifold) the associated rotation vector is non-null. And in section \ref{sectionpast} we study the relationship between rotation vectors of a homeomorphism and its inverse (which on the torus are simply opposite), and conclude that for almost all points the two are either both zero, or different (one can easily construct an example of a periodic orbit on a compact hyperbolic surface such that the two vectors aren't opposite).

As an application, we extend the concept to flows and flows on fiber bundles over complete Riemannian manifolds of non-positive curvature  (sections \ref{sectionflows} and \ref{sectionbundles}) and prove a semi-conjugacy result which improves on \cite[Theorem 4.1]{MR1802657} by relaxing the restriction of strict hyperbolicity of the base manifold.  Here, we require the base manifold to satisfy a visibility condition and the example given in section \ref{sectionpast} shows that the result may fail without this hypothesis.

In section \ref{sectionexamples} we discuss implications of our results for homeomorphisms of surfaces and illustrate them with a discussion of the `magnetic flow'.

It is noteworthy that, on one extreme, the existence theorem for rotation vectors on the two-torus is a direct consequence of Birkhoff's ergodic theorem while, on the other extreme, it was observed by the authors of \cite{MR1093216} that there doesn't seem to be an adequate ergodic theorem for establishing the existence of asymptotic homotopy cycles.

The techniques we use for establishing our main theorem arose by successive generalization of Oseledets' multiplicative ergodic theorem and were developed in \cite{MR947327},\cite{MR1729880} and \cite{MR2271477}.  Since the space of interest for Oseldets' theorem is that of positive definite symmetric real matrices (which has non-positive curvature but also flat totally geodesic submanifolds) these generalizations are well adapted to non-positively curved manifolds.  However, the results in question concern orbits of cocycles of isometries and hence are not directly applicable to other homeomorphisms (for which, for example, it might be the case that two points on the same orbit don't even have isometric neighborhoods).

In sections \ref{rate} \ref{geodesic} and \ref{geometric}, besides a standard application of Kingman's ergodic theorem, we have basically isolated Karlsson and Margulis' geometric arguments from \cite{MR1729880} and specialized them to Hadamard manifolds (as opposed to uniformly convex and Busemann non-positively curved metric spaces which was their original domain of application).  In particular a more general version of our Lemma \ref{geometriclemma} is implicit in the proof of their main theorem.  The fact that neither the result nor the relevant definitions are explicitly stated in the cited work is one reason why we include a full proof (another being that in the case of Hadamard manifolds some simplifications are possible).

In section \ref{alignment}, we prove a new ergodic theorem which is related to \cite[Theorem 1.2]{MR2271477} but applies to sequences not arising from cocycles of isometries.  This is enough to establish the main existence theorems for rotation vectors.  The arguments of the rest of the paper rely on this result and on general facts about non-positively curved manifolds which are presented in an appendix.

\section{Statement of the Main Theorem}
\label{statement}
Let us begin by reviewing the definition of a rotation vector for a homeomorphism of the torus.

\begin{definition}[Rotation Vector for Torus Homeomorphisms]
\label{rotationvectorsfortorusmaps}
Let $T^d = \mathbb{R}^d/\mathbb{Z}^d$ be the $d$-dimensional torus, and let $f: T^d \to T^d$ be a homeomorphism that is isotopic to the identity.  Also, fix a lift $F: \mathbb{R}^d \to \mathbb{R}^d$ of $f$.

The rotation vector $v_F(x)$ of a point $x \in T^d$ is the following limit in case it exists:
\[v_F(x) = \lim_{n \to +\infty}\frac{F^n(\tilde{x}) - \tilde{x}}{n}\text{ where }\tilde{x}\in \mathbb{R}^d\text{ is any lift of }x\]
\end{definition}

The following is a well known consequence of Birkhoff's Ergodic Theorem.

\begin{proposition}
\label{vectorsexistonthetorus}
If $f: T^d \to T^d$ is isotopic to the identity, $F: \mathbb{R}^d \to \mathbb{R}^d$ is a lift of $f$, and $\mu$ is an $f$-invariant Borel probability measure.  Then the limit $v_F(x)$ exists for $\mu$-almost every $x \in T^d$.
\end{proposition}

A point $x \in T^d$ has rotation vector $v_F(x)$ for a certain lift $F$ of a homeomorphism $f: T^d \to T^d$, if and only if for any lift $\tilde{x}$ of $x$ it holds that $\|F^n\tilde{x} - (\tilde{x} + nv_F(x))\| = o(n)$ when $n \to +\infty$.  If one considers the usual flat metric on $T^d$ then the fact that the curve $t \mapsto x + tv_F(x)$ is a geodesic shows that the following definition generalizes Definition \ref{rotationvectorsfortorusmaps}.

By a Riemannian covering we mean a covering map $\pi: \widetilde{M} \to M$ between Riemannian manifolds that is a local isometry at all points.  We will denote the distance function on a Riemannian manifold by $d$.

\begin{definition}[Rotation Vector]
\label{rotationvector}
Let $\pi:\widetilde{M} \to M$ be a Riemannian covering, $f: M \to M$ a homeomorphism, and $F:\widetilde{M} \to \widetilde{M}$ a lift of $f$.

A rotation vector $v_F(x)$ of a point $x \in M$ is a vector in the tangent space $T_xM$ such that for any lift $\tilde{x}$ of $x$ the following holds:
\[d(\tilde{\alpha}(n),F^n\tilde{x}) = o(n)\text{ when }n\to +\infty\]
where $\tilde{\alpha}$ is the lift starting at $\tilde{x}$ of the geodesic $\alpha:[0,+\infty) \to M$ defined by $\alpha(t) = \exp_x(tv_F(x))$.
\end{definition}

If the lift $F$ in the above definition commutes with all covering transformations then the existence of a rotation vector $v_F(x)$ is a statement about the $F$-orbit of a single lift $\tilde{x}$ of $x$.  Also, notice that in this case the expression $\rho(x) = d(\tilde{x},F\tilde{x})$ where $\tilde{x}$ is a lift of $x$ is a well defined continuous function from $M$ to $[0,+\infty)$.  Finally, notice that if $f:M \to M$ is isotopic to the identity then one can always obtain a lift $F: \widetilde{M} \to \widetilde{M}$ of $f$ that commutes with all covering transformations (e.g. this can be done by lifting the isotopy between the identity map on $M$ and $f$).

We recall that a Hadamard manifold is a complete, connected and simply connected Riemannian manifold with non-positive sectional curvature.

We can now state our main theorem as follows.

\begin{theorem}[Main Theorem]
\label{maintheorem}
Let $\widetilde{M}$ be a Hadamard manifold and $\pi:\widetilde{M} \to M$ a Riemannian covering.  For each $x \in M$ let $\tilde{x} \in \widetilde{M}$ denote an arbitrary lift of $x$.

Suppose that $f: M \to M$ is a homeomorphism that is isotopic to the identity, $F: \widetilde{M} \to \widetilde{M}$ is a lift of $f$ that commutes with all covering transformations, and $\mu$ is an $f$-invariant Borel probability measure satisfying the condition:
\[\int_M d(\tilde{x},F\tilde{x})\mathrm{d}\mu(x) < +\infty\]

Then for $\mu$-almost every $x \in M$ there exists a unique rotation vector $v_F(x) \in T_xM$.
\end{theorem}

The rotation vectors given by the above theorem depend on the choice of a lift $F$ which commutes with all covering transformations (from now on we will call such a lift `admissible').

If $M = T^d$ then it is simple to show that all lifts are admissible,  any two lifts differ by a translation with integer coordinates, and the corresponding rotation vectors satisfy the same relationship.

For hyperbolic manifolds we have the following result:
\begin{proposition}
If $M$ is a complete hyperbolic manifold with finite volume and $f:M \to M$ is isotopic to the identity then there is a unique admissible lift of $f$ to the universal Riemannian covering space $\widetilde{M}$ of $M$.
\end{proposition}
\begin{proof}
Let $F_1$ and $F_2$ be distinct admissible lifts.  It follows that $F_1^{-1}\circ F_2$ is a covering tranformation which commutes with all others.  In particular the maximal normal Abelian subgroup of $\pi_1(M)$ is non-trivial. By \cite[Theorem 10.3.10]{MR1441541} this implies that $\widetilde{M}$  splits into a Riemannian direct product isometric to $\mathbb{R} \times X$ for some simply connected Riemannian manifold $X$.  However by taking two distinct points $x_1,x_2 \in X$ and considering the curves $t \mapsto (t,x_i)$ we see that this would imply the existence of disjoint geodesics which remain at a positive fixed distance.  This contradicts the fact that $\widetilde{M}$ is isometric to hyperbolic space.
\end{proof}

We observe that because its proof proceeds by contradicting the strong visibility property (see Appendix) the proposition is valid if this property is assumed for $M$ instead of hyperbolicity.

In general, assuming that $M$ is complete and has finite volume one obtains that either the conclusion of the above proposition holds or $\widetilde{M}$ splits into a Riemannian direct product with a non-trivial factor isometric to $\mathbb{R}^d$.  Furthermore if $M$ is compact it follows from \cite[The Center Theorem]{MR0334083} that $M$ is foliated by isometrically embdedded flat tori.  This suggests that a characterization of the dependence of rotation vectors on the choice of lift similar to the case in which $M = T^d$ might be possible.  However, such a result is unknown to the author at the time of writing.

\section{Rate of Escape}
\label{rate}

\begin{definition}[Rate of Escape]
\label{rateofescape}
The rate of escape of a sequence $\{x_n\}_{n \ge 0} \subset X$ in a metric space $X$ is the value of the following limit if it exists:
\[R = \lim_{n \to +\infty}\frac{d(x_0,x_n)}{n}\]
\end{definition}

Observe that if $M,\widetilde{M},f$ and $F$ are as in the statement of Theorem \ref{maintheorem} and there exists a rotation vector $v_F(x)$ for a certain $x \in M$ then for any lift $\tilde{x}$ of $x$ the sequence $\{F^n\tilde{x}\}_{n\ge 0}$ has rate of escape $\|v_F(x)\|$ (for details see Lemma \ref{vectorsvrsescorts}).  In view of this our first task is to show that for almost every $x \in M$ the sequence $\{F^n\tilde{x}\}_{n\ge 0}$ has a well defined finite rate of escape.

The purpose of this section is to show a somewhat more general fact, i.e. that a certain general class of random sequences in a metric space almost surely have a well defined finite rate of escape.  This is a consequence of Kingman's Subadditive Ergodic Theorem which we restate in a convenient fashion.

\begin{theorem}[Kingman's Subadditive Ergodic Theorem \cite{MR0254907}]
\label{kingmanstheorem}
If $(M,\mathcal{B},\mu)$ is a probability space, $f:M \to M$ is a measure preserving measurable function, and $\{a(m,n)\}_{0 \le m < n}$ (where $m,n \in \mathbb{Z}$) is a family of measurable functions from $M$ to $\mathbb{R}$ satisfying:
\begin{enumerate}
\item $a_x(l,n) \le a_x(l,m) + a_x(m,n)$ for all $x \in M$ and all $0 \le l < m < n$
\item $a_x(m+1,n+1) = a_{f(x)}(m,n)$ for all $x \in M$ and all $0 \le m < n$
\item $\int_M |a_x(0,1)| \mathrm{d}\mu(x) < +\infty$
\item $\lim_{n \to +\infty} \int_M \frac{a_x(0,n)}{n} \mathrm{d}\mu(x) > -\infty$
\end{enumerate}
Then there exists an invariant and integrable function $R: M \to \mathbb{R}$ such that
\[\int_A R(x) \mathrm{d}\mu(x) = \lim_{n \to +\infty}\int_A \frac{a_x(0,n)}{n} \mathrm{d}\mu(x)\]
for every invariant measurable set $A \subset M$.  And for $\mu$ almost every $x \in M$ the following holds:
\[ R(x) = \lim_{n \to + \infty}\frac{a_x(0,n)}{n}\]
\end{theorem}

A family of measurable functions satisfying condition $1$ is called a subadditive process.  Condition $2$ guarantees that this process is stationary.  Conditions $3$ and $4$ imply in particular that each function in the process is integrable (this is a consequence of subadditivity).  Processes satisfying condition $4$ are said to have ``finite time constant'', one can prove Kingman's Theorem without this condition but the pointwise limit $R$ will no longer be integrable (e.g. it might be equal to $-\infty$ on a set of positive probability).  Since its first proof in \cite{MR0254907}, the theorem has received many alternative proofs.  See \cite{MR797411} for a general reference.

A typical example of a subadditive process arising from a dynamical system is the following: Let $f: M \to M$ be a volume preserving diffeomorphisms of a compact Riemannian manifold $M$ and define $a_x(m,n) = \log(\|D_{f^m(x)}f^{n-m}\|)$.  In this example Kingman's theorem gives the existence of the largest Lyapunov exponent for almost every orbit (which was previously an independent result first proved by Furstenberg and Kesten, see \cite[Introduction]{MR2271477}).

We will use Kingman's theorem in the following form.

\begin{corollary}[Rate of Escape for Random Sequences]
\label{rateofescapeforhomeos}
Suppose $(M,\mathcal{B},\mu)$ is a probability space and $f: M \to M$ is a measurable and measure preserving transformation.

Let $X$ be a metric space and $\phi: M \to X^{\mathbb{N}}$ a measurable function such that the family of functions $\{a(m,n)\}_{n > m \ge 0}$ that is defined by
\[a(m,n): M \to [0,+\infty)\]
\[a_x(m,n) = d(\phi(x)_m,\phi(x)_n)\]
satisfies the hypothesis of Theorem \ref{kingmanstheorem}.  Then there exists an invariant and integrable function $R: M \to [0,+\infty)$ such that
\[\int_AR(x)\mathrm{d}\mu(x) = \lim_{n \to +\infty}\int_A\frac{a_x(0,n)}{n}\mathrm{d}\mu(x)\]
for every invariant measurable set $A \subset M$.  And for almost every $x \in M$ the following holds:
\[R(x) = \lim_{n \to +\infty}\frac{a_x(0,n)}{n}\]
In particular, for almost every $x \in M$ the sequence $\phi(x)$ has a finite rate of escape.
\end{corollary}

\section{Geodesic Escorts}
\label{geodesic}

The main theorems of \cite{MR947327} and \cite{MR1729880} can be restated in terms of the following definition.  We will prove in this section that it is also connected to rotation vectors.

\begin{definition}[Geodesic Escort]
\label{geodesicescort}
Let $(X,d)$ be a metric space.  A sequence $\{x_n\}_{n \ge 0} \subset X$ is said to be escorted by a geodesic, if there exists a function $\alpha:[0,+\infty) \to X$ that is either constant or a local isometry onto its image and satisfies:
\begin{itemize}
\item $\alpha(0) = x_0$
\item $d(x_n, \alpha(d(x_0,x_n))) = o(n)\text{ when }n \to +\infty$
\end{itemize}
\end{definition}

Notice that any sequence with rate of escape equal to zero is escorted by a (constant) geodesic.

Before giving a proof of a necessary and sufficient condition for the existence of a rotation vector we would like to recall the following facts:
\begin{itemize}
\item If $\widetilde{M}$ is a Hadamard manifold then it is diffeomorphic to $\mathbb{R}^{\text{dim}(\widetilde{M})}$ (see for example \cite{MR1666820} Theorem 3.8 on page 252).
\item Any two points in $\widetilde{M}$ belong to a unique geodesic (up to reparametrizations) and therefore any arclength parametrization of this geodesic is a (global) isometry onto its image (see \cite{MR1666820} p.353 Corollary 3.11).
\end{itemize}

\begin{lemma}
\label{vectorsvrsescorts}
Let $\widetilde{M}$ be a Hadamard manifold and $\pi:\widetilde{M} \to M$ a Riemannian covering.

Suppose that $f: M \to M$ is a homeomorphism that is isotopic to the identity, and that $F: \widetilde{M} \to \widetilde{M}$ is a lift of $f$ that commutes with all covering transformations.

A rotation vector $v_F(x)$ exists for a point $x \in M$ if and only if there exists a lift $\tilde{x}$ of $x$ such that $\{F^n\tilde{x}\}_{n \ge 0}$ has a well defined rate of escape and is escorted by a geodesic.

In such a case it holds that for any lift $\tilde{x}$ of $x$ the sequence $\{F^n\tilde{x}\}$ is escorted by a geodesic and has rate of escape $\|v_F(x)\|$.
\end{lemma}
\begin{proof}
First suppose $v = v_F(x) \in T_xM$ is a rotation vector for $x$.  For any lift $\tilde{x}$ of $x$ the geodesic $\tilde{\alpha}:[0,+\infty) \to \widetilde{M}$ given by Definition \ref{rotationvector} satisfies:
\[d(\tilde{\alpha}(n),F^n\tilde{x}) = o(n)\text{ when }n\to+\infty\]

Since each pair of points in $\widetilde{M}$ belongs to a unique geodesic we have that:
\[d(\tilde{\alpha}(0),\tilde{\alpha}(n)) = d(\tilde{x},\tilde{\alpha}(n)) = n\|\tilde{\alpha}'(0)\| = n\|v\|\]

The triangle inequality now implies:
\[d(\tilde{x},\tilde{\alpha}(n)) - d(\tilde{\alpha}(n),F^n\tilde{x}) \le d(\tilde{x},F^n\tilde{x}) \le d(\tilde{x},\tilde{\alpha}(n)) + d(\tilde{\alpha}(n),F^n\tilde{x})\]

And this in turn implies:
\[n\|v\| - o(n) \le d(\tilde{x},F^n\tilde{x}) \le n\|v\| + o(n)\]

Which shows that $\{F^n\tilde{x}\}_{n \ge 0}$ has rate of escape $\|v\|$.

If $\|v\| = 0$ then $\alpha$ is constant and is a geodesic escort for $\{F^n\tilde{x}\}$, otherwise let $\beta: [0,+\infty) \to \widetilde{M}$ be the geodesic starting at $\tilde{x}$ with $\beta'(0) = \tilde{\alpha}'(0)/\|\tilde{\alpha}'(0)\|$.  Since 
\[\beta(d(\tilde{x},F^n\tilde{x})) = \beta(n\|v\| + o(n)) = \tilde{\alpha}(n + o(n))\]
one has that:
\[d(F^n\tilde{x},\beta(d(\tilde{x},F^n\tilde{x})) \le d(F^n\tilde{x},\tilde{\alpha}(n)) + d(\tilde{\alpha}(n),\tilde{\alpha}(n + o(n)))= o(n)\]
and therefore $\beta$ escorts the sequence $\{F^n\tilde{x}\}_{n\ge 0}$.

Suppose now that for some lift $\tilde{x}$ of $x$ one has that $\{F^n\tilde{x}\}_{n \ge 0}$ has rate of escape $R$ and is escorted by a geodesic $\beta$.  Let $\gamma:\widetilde{M} \to \widetilde{M}$ be a covering transformation.  Since $F$ commutes with $\gamma$ and $\gamma$ is an isometry one has that $\gamma\circ \beta$ is a geodesic escort for the sequence $\{F^n(\gamma\tilde{x})\}_{n \ge 0}$ and this sequence has rate of escape $R$.  In particular the rate of escape is independent of the chosen lift $\tilde{x}$, and therefore if $R = 0$ then $v = 0 \in T_xM$ is a rotation vector for $x$.

On the other hand if $R > 0$, then $\beta$ is a non-constant geodesic and the image $v$ of the vector $R\beta'(0)$ under the differential of the covering map $\pi: \widetilde{M} \to M$ is independent of the chosen lift $\tilde{x}$.  To show that $v$ is a rotation vector for $x$ all that is needed is to prove that $d(\beta(Rn),F^n\tilde{x}) = o(n)$.  We can obtain this directly from $d(\tilde{x},F^n\tilde{x}) = Rn + o(n)$ and the fact that $\beta$ is a geodesic escort, as follows:
\[d(\beta(Rn),F^n\tilde{x}) \le d(\beta(Rn),\beta(d(\tilde{x},F^n\tilde{x}))) + d(\beta(d(\tilde{x},F^n\tilde{x})),F^n\tilde{x}) = o(n)\]
\end{proof}

\section{Aligned Sequences}
\label{geometric}

\begin{definition}[Linear Escape to Infinity]
A sequence $\{x_n\}_{n \ge 0} \subset X$ in a metric space $X$ is said to escape linearly to infinity if it has a positive and finite rate of escape.
\end{definition}

In this section we will give a condition under which a sequence that escapes linearly to infinity will be escorted by a geodesic.

The simplest such condition known to the author is the following (which we will state without proof), valid for $d$-dimensional hyperbolic space:
\begin{proposition}
Let $\mathbb{H}^d$ denote $d$-dimensional hyperbolic space.  Any sequence $\{x_n\}_{n \ge 0} \subset \mathbb{H}^d$ that escapes linearly to infinity and satisfies 
\[d(x_n,x_{n+1}) = o(n)\text{ when }n\to +\infty\]
is escorted by a unique geodesic.
\end{proposition}

The condition $d(x_n,x_{n+1}) = o(n)$ is almost always satisfied by random sequences, provided that the variables $d(x_n,x_{n+1})$ are identically distributed and have finite expectation.  Using this fact one can obtain a proof of a special case of our main theorem.

However the sequence in $\mathbb{C}$ defined by $x_n = ne^{i\log(n)}$ satisfies $d(x_n,x_{n+1}) = o(n)$ (where $d(x,y) = |x-y|$ for every $x,y \in \mathbb{C}$) but isn't escorted by a geodesic.  This shows that the proposition is false for general non-positively curved spaces.

Since our objective is to prove the existence of geodesic escorts for sequences in Hadamard manifolds we will need a stronger hypothesis than $d(x_n,x_{n+1}) = o(n)$.  The following definitions will allow us to formulate such a hypothesis.

\begin{definition}[$\epsilon$-Cone]
Let $(X,d)$ be a metric space.  If $\epsilon \in [0,+\infty]$ and $x,y \in X$ the $\epsilon$-cone from $x$ to $y$ is defined as the following set:
\[[x,y]_\epsilon = \{z \in X: e^{-\epsilon}d(x,z) + d(z,y) \le d(x,y)\}\]
\end{definition}

The $0$-cone between two points in $\mathbb{C}$ is a segment.  The cone $[x,y]_{+\infty}$ is a closed disk centered at $y$ and containing $x$ (in particular note that the definition is not symmetric in $x$ and $y$).

\begin{figure}[htp]
\centering
\includegraphics[totalheight=0.3\textheight]{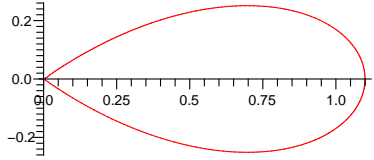}
\caption{The boundary of $[0,1]_{0.2}$ in $\mathbb{C}$.}
\end{figure}


We will now show that in Hadamard manifolds the $\epsilon$-cone between two points is close to a geodesic segment for small $\epsilon$.

We will rely on two properties of Hadamard manifolds which we state below without proof.  Lemma \ref{convexity} is contained in theorems 4.3 and 4.4 of chapter IX in \cite{MR1666820}.   While the semi-parallelogram law is proved to hold locally in section 3 of chapter XI of the same reference.  In the context of Hadamard manifolds the same proof gives the (global) statement below.

\begin{lemma}[Convexity of geodesic distance]
\label{convexity}
If $M$ is a Hadamard manifold then for any $x \in M$ and any pair of geodesics $\alpha,\beta:\mathbb{R} \to M$ the following two functions are convex:
\[t \mapsto d(x,\alpha(t))\]
\[t \mapsto d(\alpha(t),\beta(t))\]
\end{lemma}

\begin{lemma}[Semi-Parallelogram Law]
\label{semiparallelogramlaw}
Let $M$ be a Hadamard manifold and let $x,y,z \in M$.  Then if $m$ is the midpoint of the geodesic segment $[x,y]_0$ the following inequality holds:
\[d(x,y)^2 + 4d(m,z)^2 \le 2d(x,z)^2 + 2d(z,y)^2\]
\end{lemma}

\begin{lemma}
\label{epsilonconelemma}
Let $M$ be a Hadamard manifold.  For all $\epsilon > 0$ and all $x,y,z \in M$ with $z \in [x,y]_\epsilon$ it holds that:
\[d(z,w)^2 \le 4(1- e^{-2\epsilon})d(x,z)^2\]
where $w = \alpha(d(x,z))$ and $\alpha:[0,+\infty) \to M$ is the unique geodesic parametrized by arclength with $\alpha(0) = x$ and $\alpha(d(x,y)) = y$.
\end{lemma}
\begin{proof}
Let $a = d(x,z) = d(x,w), b = d(x,y)-a$ and $c = d(z,y)$.  Notice that $b = d(w,y)$ if $d(x,z) < d(x,y)$ and $b = -d(w,y)$ otherwise.  In both cases $|b| \le c$.

Since $z \in [x,y]_\epsilon$ we have:
\[e^{-\epsilon}a + c \le a+b\]
which implies
\[a+b-c \ge e^{-\epsilon}a\]

Let $m$ be the midpoint of the segment $[z,w]_0$.  Lemma \ref{convexity} implies that:
\[d(y,m) \le \max(|b|,c) = c\]
and from this we obtain
\[d(x,m) \ge d(x,y) - d(y,m) \ge a+b-c \ge e^{-\epsilon}a\]

The semi-parallelogram law (lemma \ref{semiparallelogramlaw}) now gives:
\[d(w,z)^2 + 4d(x,m)^2 \le 4a^2\]

Combining these inequalities we obtain:
\[d(w,z)^2 \le 4(a^2 - d(x,m)^2) \le 4(1 - e^{-2\epsilon})a^2 = 4(1 - e^{-2\epsilon})d(x,z)^2\]
\end{proof}

We will now state a condition that will guarantee the existence of a geodesic escort for a sequence that escapes linearly to infinity.  One can interpret the condition saying either that the sequence eventually stays in arbitrarily small $\epsilon$-cones, or that the differences $d(x_0,x_n)-d(x_k,x_n)$ are close to their largest possible value (i.e. $d(x_0,x_k)$).

\begin{definition}[Aligned Sequence]
Let $(X,d)$ be a metric space. A sequence $\{x_n\}_{n\ge 0} \subset X$ is said to be aligned if for each $\epsilon > 0$ there exists $K \in \mathbb{N}$ such that for infinitely many $n \in \mathbb{N}$ the following holds:
\[x_k \in [x_0,x_n]_{\epsilon}\text{ for all }K \le k \le n\]
\end{definition}

\begin{observation}
\label{geometriclemmaobservation}
Let $(X,d)$ be a metric space and $\{x_n\}_{n \ge 0} \subset X$ a sequence that escapes linearly to infinity with rate of escape $R$.  Then $\{x_n\}_{n \ge 0}$ is aligned if and only if $L \ge R$ where:
\[L = \lim_{K \to +\infty}\limsup_{n \to +\infty}\min_{K \le k \le n}\{\frac{d(x_0,x_n)-d(x_n,x_k)}{k}\}\]
\end{observation}

\begin{lemma}
\label{geometriclemma}
Let $H$ be a Hadamard manifold.  If $\{x_n\}_{n \ge 0} \subset H$ escapes linearly to infinity and is aligned then it is escorted by a geodesic.
\end{lemma}
\begin{proof}
Let $f:[0,+\infty) \to [0,+\infty)$ be given by $f(\epsilon) = 2\sqrt{1- e^{-2\epsilon}}$.

For each $n \in \mathbb{N}$ define $d_n = d(x_0,x_n)$ and let $\alpha_n: [0,+\infty) \to H$ be the unique geodesic parametrized by arclength such that $\alpha_n(0) = x_0$ and $\alpha_n(d_n) = x_n$.  Also, let $R  > 0$ be the rate of escape of the sequence $\{x_n\}_{n \ge 0}$.

By hypothesis for each $\epsilon > 0$ there is a natural number $K_{\epsilon}$ and infinitely many values of $n$ (which we will call $\epsilon$-admissible) such that:
\[x_k \in [x_0,x_n]_\epsilon\ \text{ for all } K_{\epsilon} \le k \le n\]

Since $d_k = Rk + o(k)$ we may assume that $K_{\epsilon}$ above is chosen large enough so that $d_k \ge 1$ for all $k \ge K_{\epsilon}$.

Also, note that by lemma \ref{epsilonconelemma} we have the following inequality for all $\epsilon$-admissible $n$ and all $k$ with $K_{\epsilon} \le k \le n$:
\[d(\alpha_n(d_k),x_k) \le f(\epsilon)d_k\]

We will prove the lemma in two steps:  First, we will show that $\{\alpha_n(1)\}_{n \ge 0}$ is a Cauchy sequence.  Second, defining $\alpha$ to be the geodesic ray with $\alpha(0) = x_0$ and $\alpha(1) = \lim_{n \to +\infty}\alpha_n(1)$ we will show that $\alpha$ is a geodesic escort for $\{x_n\}_{n \ge 0}$.

To prove the first claim fix $\epsilon > 0$ and the corresponding $K_{\epsilon}$.  For any $k,l \ge K_{\epsilon}$ we can find an $\epsilon$-admissible $n \ge \max(k,l)$ and for this value of $n$ one has:
\begin{align*}
d(\alpha_k(1),\alpha_l(1)) &\le d(\alpha_k(1),\alpha_n(1)) + d(\alpha_n(1),\alpha_l(1))
\\ &\le \frac{d(x_k,\alpha_n(d_k))}{d_k} + \frac{d(\alpha_n(d_l),x_l)}{d_l}
\\ &\le 2f(\epsilon)
\end{align*}
where the second inequality is obtained by applying lemma \ref{convexity} to $\alpha_n$ and either $\alpha_k$ (in the case of the first summand) or $\alpha_l$ (for the second).  Here we have used the fact that $d_k,d_l \ge 1$.

As $\epsilon$ can be chosen so that $f(\epsilon)$ is arbitrarily small the above argument shows that $\alpha_n(1)$ is a Cauchy sequence and therefore $\lim_{n \to +\infty}\alpha_n(1)$ exists.  Therefore one can define a geodesic ray $\alpha$ as claimed above.

We will now show that $\alpha$ is a geodesic escort for $\{x_n\}_{n \ge 0}$.  

Since $\alpha(0) = \alpha_n(0)$ for all $n$ and $\alpha(1) = \lim_{n \to +\infty}\alpha_n(1)$ it follows that
\[\alpha(t) = \lim_{n \to +\infty}\alpha_n(t)\]
for all $t$\begin{footnote}{By the Arsel\`a-Ascoli theorem any subsequence of $\alpha_n$ has a subsequence converging uniformly on compact sets.  If $\beta$ is the limit of such a subsequence then it must be a geodesic joining $\alpha(0)$ and $\alpha(1)$ and it follows that $\alpha = \beta$.}\end{footnote}.

Given $\epsilon > 0$ and the corresponding $K_{\epsilon}$ for each $k \ge K_{\epsilon}$ one can choose a sequence $n_l \to +\infty$ of $\epsilon$-admissible numbers.  For this sequence we obtain:
\begin{align*}
d(\alpha(d_k),x_k) = \lim_{l \to +\infty}d(\alpha_{n_l}(d_k),x_k) \le f(\epsilon)d_k
\end{align*}
for all $k \ge K_{\epsilon}$.

Hence for each $\epsilon > 0$ one has $d(\alpha(d_k),x_k) \le f(\epsilon)d_k$ for all $k \ge K_{\epsilon}$.  This shows that:
\[d(\alpha(d_k),x_k) = o(d_k)\text{ when }k \to +\infty\]
so that $\alpha$ is a geodesic escort as claimed.
\end{proof}

We conclude this section by showing that if an orbit of an isometry or semi-contraction escapes linearly to infinity, then it is aligned.  This result isn't necessary for the proof of our main theorem, but will be used later on as a source of examples.

\begin{lemma}
\label{alignmentoforbitsofsemicontractions}
Let $(X,d)$ be a metric space and $\{x_n\}_{n \ge 0} \subset X$ a sequence that escapes linearly to infinity and satisfies:
\[d(x_{m+k},x_{n+k}) \le d(x_m,x_n) \text{ for all } k,m,n \ge 0\]

Then $\{x_n\}_{n \ge 0}$ is aligned.
\end{lemma}
\begin{proof}
Let $R$ be the rate of escape of $\{x_n\}$. 

For each $\epsilon > 0$ take $K$ such that:
\[d(x_0,x_k) \le e^{\epsilon/2}Rk\text{ for all } k \ge K\]

Since the sequence $d(x_0,x_n) - e^{-\epsilon/2}Rn$ is unbounded there exist infinitely many $n$ such that:
\[d(x_0,x_n) - e^{-\epsilon/2}Rn > d(x_0,x_m) - e^{-\epsilon/2}Rm \text{ for all } m < n\]

Hence for infinitely many $n$ it holds that for all $K \le k \le n$:
\[e^{-\epsilon}d(x_0,x_k) + d(x_k,x_n) \le e^{-\epsilon/2}Rk + d(x_0,x_{n-k}) < d(x_0,x_n)\]

This implies that $x_k \in [x_0,x_n]_\epsilon$ and hence the sequence is aligned.
\end{proof}

\section{Alignment of Random Sequences}

\label{alignment}

The following lemma will enable us to prove that, in a sense, almost all random sequences that escape linearly to infinity are aligned.  

\begin{lemma}[\cite{MR1729880} Lemma 4.1]
\label{karlssonmargulis}
Suppose $\{a(m,n)\}_{0 \le m < n}$ satisfies the hypothesis of Theorem \ref{kingmanstheorem} and that:
\[\lim_{n \to +\infty}\int_M \frac{a_x(0,n)}{n}\mathrm{d}\mu(x) > 0\]

Then the probability of the following set is strictly positive:
\[\{x \in M: \text{ for infinitely many }n\in \mathbb{N}, a(0,n) - a(k,n) > 0\text{ for all } 1 \le k \le n\}\]
\end{lemma}
\begin{proof}
Choose $\epsilon$ positive but smaller than the time constant of $\{a(m,n)\}$ and consider the subadditive cocycle $c: \mathbb{N} \times M \to \mathbb{R}$ (over $f: M \to M$) given by:
\[c(n,x) = a_x(0,n) - \epsilon n\text{ for all }x \in M, n \in \mathbb{N}\]

The cocycle $c$ satisfies the hypothesis of Lemma 4.1 in \cite{MR1729880} and the conclusion follows.
\end{proof}

We will now establish the ergodic theorem used in the proofs of existence of rotation vectors throughout this work.

\begin{theorem}[Alignment of Random Sequences]
\label{alignmentofrandomsequences}
Suppose $(M,\mathcal{B},\mu)$ is a probability space and $f: M \to M$ is a measurable and measure preserving transformation.

Let $X$ be a metric space and $\phi: M \to X^{\mathbb{N}}$ a measurable function such that the family of functions $\{a(m,n)\}_{n > m \ge 0}$ defined by:
\[a(m,n): M \to [0,+\infty)\]
\[a_x(m,n) = d(\phi(x)_m,\phi(x)_n)\]
satisfies the hypothesis of Theorem \ref{kingmanstheorem}.  Then for almost every $x \in M$ if the sequence $\phi(x)$ escapes linearly to infinity then it is aligned.
\end{theorem}
\begin{proof}
Let $R: M \to [0,+\infty)$ be the function given by Corollary \ref{rateofescapeforhomeos}.  For almost every $x \in M$, the sequence $\phi(x)$ escapes linearly to infinity if and only if $R(x) > 0$.

Consider the measurable function $L: M \to [0,+\infty)$ defined by:
\[L(x) = \max(0,\lim_{K \to +\infty}\limsup_{n \to +\infty}\min_{K \le k \le n}\{\frac{d(\phi(x)_0,\phi(x)_n) - d(\phi(x)_n,\phi(x)_k)}{k}\})\]

By Observation \ref{geometriclemmaobservation} it suffices to show that $L \ge R$ on a set of full measure.

First we will show that the function $L$ is invariant on a set of full measure.  For this purpose it suffices to show that $L(x) \le L(f(x))$ for almost every $x$, since the probability of the set $\{x \in M: L(x) \ge a\}$ is equal to that of the set $\{x \in M: L(f(x)) \ge a\}$ for every $a \in \mathbb{R}$ because $f$ is measure preserving.

If $L(x) = 0$ then trivially $L(x) \le L(f(x))$.  Suppose $L(x) \ge t > 0$.  This implies that for each $\epsilon > 0$ there exists $K \in \mathbb{N}$ and infinitely values of $n \in \mathbb{N}$ such that:
\[d(\phi(x)_0,\phi(x)_n) - d(\phi(x)_n,\phi(x)_k) > te^{-\epsilon}k\text{ for all }K \le k \le n\]
for a slightly larger $K$ the following holds:
\begin{align*}
d(\phi(x)_1,\phi(x)_n) &- d(\phi(x)_n,\phi(x)_k) \ge  -d(\phi(x)_1,\phi(x)_0) + d(\phi(x)_0,\phi(x)_n) - d(\phi(x)_n,\phi(x)_k)\\
&> -d(\phi(x)_1,\phi(x)_0) +te^{-\epsilon}k > te^{-2\epsilon}(k-1)\text{ for all }K \le k \le n
\end{align*}
and therefore $L(f(x)) \ge te^{-2\epsilon}$.  Since this holds for every $\epsilon > 0$ we have shown the claim that $L(x) \le L(f(x))$ almost everywhere and therefore $L$ is invariant on a set of full measure.  

To simplify the discussion, modify $L$ on a set of measure $0$ so that it is strictly invariant.

We will now show that $L \ge R$ almost everywhere. 

This is trivially true in the set where $R = 0$.

Suppose that for some $\epsilon > 0$ the set $A = \{x \in M: L(x) < e^{-\epsilon}R(x)\}$ has positive measure.  Let $1_A$ be the function that takes the value $1$ on $A$ and $0$ outside of $A$, and consider the stationary subadditive process $\{b(m,n)\}_{0 \le m < n}$ defined by
\[b_x(m,n) = (d(\phi(x)_m,\phi(x)_n) - (n-m)e^{-\epsilon}R(x))1_A(x)\]

Since this process satisfies
\[\lim_{n \to +\infty}\int_M \frac{b_x(0,n)}{n}\mathrm{d}\mu(x) = (1-e^{-\epsilon})\int_AR(x)\mathrm{d}\mu(x) > 0\]
by Lemma \ref{karlssonmargulis} there is a set of positive probability of $x \in M$ such that there exist infinitely many $n$ satisfying the following:
\[b_x(0,n) - b_x(k,n) > 0\text{ for all }1 \le k \le n\]

However since
\[b_x(0,n) - b_x(k,n) = (d(\phi(x)_0,\phi(x)_n) - d(\phi(x)_n,\phi(x)_k) - ke^{-\epsilon}R(x))1_A(x)\]
this would imply that for some $x \in A$ we have $L(x) \ge e^{-\epsilon}R(x)$ contradicting the definition of $A$.  Therefore we must have $L \ge R$ almost everywhere as claimed.
\end{proof}

\section{Proof of the Main Theorem}
\label{proof}

We will now prove the main theorem (Theorem \ref{maintheorem}).

\begin{proof}
One can choose $\tilde{x}$ so that it is a measurable function of $x$ (e.g. apply \cite[Theorem 12.16 page 78]{MR1321597}).  By Corollary \ref{rateofescapeforhomeos}, for almost every $x \in M$ the sequence $\{F^n\tilde{x}\}_{n \ge 0}$ has a finite rate of escape $R(x)$.

If $R(x) = 0$ then $0 \in T_xM$ is a rotation vector for $x$.

On the other hand, by Theorem \ref{alignmentofrandomsequences}, for almost every $x \in M$ in the case that $R(x) > 0$ the sequence $\{F^n\tilde{x}\}_{n \ge 0}$ is aligned.  In this case the existence of a geodesic escort follows from Lemma \ref{geometriclemma}.  This implies the existence of a rotation vector by Lemma \ref{vectorsvrsescorts}.

The uniqueness claim follows from the fact that no two geodesics $\tilde{\alpha},\tilde{\beta}:[0,+\infty) \to \widetilde{M}$ with the same starting point satisfy $d(\tilde{\alpha}(n),\tilde{\beta}(n)) = o(n)$ when $n \to +\infty$.  This is a direct consequence of Lemma \ref{convexity}.
\end{proof}

\section{Periodic points and Homological Rotation Vectors}
\label{periodic}
The purpose of this section is to clarify the relationship between homological rotation vectors and rotation vectors as defined in this paper.  We will consider rotation vectors associated to periodic points.  The result of this section is illustrated by the following two examples:

\begin{itemize}
\item Let $\widetilde{M} = \{(x,y) \in \mathbb{R}^2: y > 0\}$ with the hyperbolic metric, and define $F: \widetilde{M} \to \widetilde{M}$ by $F(x,y) = (x+1,y)$.  If $M$ is any quotient of $\widetilde{M}$ by the group generated by a non-trivial translation with respect to the $x$-axis, $f$ is the projection of $F$ to $M$ and $\mu$ is an invariant probability measure supported on the projection of $\mathbb{R}\times \{1\}$ to $M$ then the hypothesis of Theorem \ref{maintheorem} are satisfied.  Direct calculation shows that the rate of escape is equal to $0$, however the homological rotation vector of every $f$ orbit in $M$ is non-null.
\item Let $M$ be the two dimensional sphere with two handles (i.e. a double-torus) with a hyperbolic metric.  There is a simple closed curve (separating both handles) on $M$ that is not homotopic to a constant but is homologically trivial.  Consider a flow $\phi: \mathbb{R} \times M \to M$ with this curve as a periodic orbit, take $f = \phi^1$ and $\mu$ to be the invariant measure for the flow supported on the selected periodic orbit.  By lifting this flow to the universal covering space and taking the time $1$ to be $F$ we obtain an example with null homological rotation vectors but with positive rate of escape. 
\end{itemize}

In this section we will have to distinguish between the usual relationship of (end-point fixing) homotopy between curves and free-homotopy equivalence between closed curves (where there is no fixed base-point).  We recall that two continuous closed curves $\alpha,\beta: [0,1] \to M$ in a manifold are called freely-homotopic if there exists a continuous function $\gamma:[0,1]\times [0,1] \to M$ such that:
\begin{enumerate}
\item $\gamma(0,t) = \alpha(t)$ for all $t$.
\item $\gamma(1,t) = \beta(t)$ for all $t$.
\item $\gamma(s,0) = \gamma(s,1)$ for all $s$.
\end{enumerate}
The requirement that both curves be parametrized on $[0,1]$ is non-essential and is removed by declaring that two arbitrary closed curves $\alpha:[a,b] \to M$ and $\beta:[c,d]\to M$ are freely-homotopic if there exist continuous and increasing reparametrizations of them satisfying the above definition. 

Free homotopy is an equivalence relationship among all closed continuous closed curves in $M$ (see \cite[Section 17]{MR807945} for a general reference on this subject).  The equivalence classes of this relationship are called free homotopy classes.

On a Riemannian manifold each free homotopy class of closed curves has an associated length as follows:
\begin{definition}[Length of a free homotopy class]
Let $M$ be a Riemannian manifold and $C$ be a free homotopy class of closed curves in $M$.  The length $l(C)$ of $C$ is defined as:
$l(C) = \inf\{|\alpha|: \alpha \in C\}$
where $|\alpha|$ denotes the length of a smooth curve $\alpha$.
\end{definition}

We will need to use the following two facts about free homotopy classes of manifolds of non-positive curvature.

\begin{lemma}\label{freehomotopyisconjugation}
Let $M$ be a connected manifold.  Two closed curves $\alpha,\beta: [0,1] \to M$ are freely homotopic if and only if there exists $\gamma:[0,1] \to M$ with $\gamma(0) = \alpha(0)$ and $\gamma(1) = \beta(0)$ such that $\alpha$ is homotopic to $\gamma^{-1}\cdot \beta \cdot \gamma$ (where the dot denotes concatenation and $\gamma^{-1}$ is an orientation-reversing reparametrization of $\gamma$).
\end{lemma}
\begin{proof}
First suppose that $\alpha$ is homotopic to $\gamma^{-1}\cdot \beta \cdot \gamma$ for some $\gamma$.  It follows that $\alpha$ is freely homotopic to $\beta\cdot \gamma \cdot \gamma^{-1}$ and hence to $\beta$.

Now suppose that $\alpha$ and $\beta$ are freely homotopic.

Choose any $\gamma_1:[0,1] \to M$ with $\gamma_1(0) = \alpha(0)$ and $\gamma_1(1) = \beta(0)$.  Notice that $\gamma_1^{-1}\cdot \beta\cdot \gamma_1$ is freely homotopic to $\beta$ and hence to $\alpha$. By \cite[Theorem 17.3.1]{MR807945} there exists $\gamma_2:[0,1] \to M$ with $\gamma_2(0) = \gamma_2(1) = \alpha(0)$ such that $\gamma_2^{-1}\cdot \gamma_1^{-1}\cdot \beta \cdot \gamma_1 \cdot \gamma_2$ is homotopic to $\alpha$.  Hence by taking $\gamma = \gamma_1 \cdot \gamma_2$ we have shown that $\alpha$ is homotopic to $\gamma^{-1}\cdot \beta \cdot \gamma$.
\end{proof}

\begin{lemma}\label{lengthofaclass}
Let $\pi: \widetilde{M} \to M$ be a Riemannian covering where $\widetilde{M}$ is a Hadamard manifold.  Let $C$ be a free homotopy class in $M$ and $\rho: \widetilde{M} \to \widetilde{M}$ be a covering transformation such that the geodesic segment $[p,\rho(p)]$ projects to a curve in class $C$ for some $p \in \widetilde{M}$.  Then the length of the class $C$ is given by:
\[l(C) = \inf\{d(x,\rho(x)): x \in \widetilde{M}\}\] 
\end{lemma}
\begin{proof}
Let $\alpha:[0,1] \to M$ be the projection of a parametrization the geodesic segment $[p,\rho(p)]$ to $M$.  If $\beta: [0,1] \to M$ is a closed curve freely homotopic to $\alpha$ then there exists $\gamma:[0,1] \to M$ such that $\alpha$ is homotopic to $\gamma^{-1}\cdot \beta \cdot \gamma$.  Hence by lifting $\gamma$ starting at $p$ the other endpoint $q$ satisfies that the segment $[q,\rho(q)]$ projects to a curve homotopic to $\beta$.  Hence the length of $\beta$ is greater than or equal to $d(q,\rho(q))$ and it follows that:
\[l(C) \ge \inf\{d(x,\rho(x)): x \in \widetilde{M}\}\]

On the other hand for any $x \in \widetilde{M}$ the geodesic segment $[x,\rho(x)]$ projects to a curve $\beta$ which is freely homotopic to $\alpha$ as one can see by considering the projection $\gamma$ of any curve $\gamma$ between $p$ and $x$ and noting that $\gamma^{-1}\cdot\beta\cdot\gamma$ is homotopic to $\alpha$.  Hence one obtains that $d(x,\rho(x))$ is greater than or equal to $l(C)$, which establishes the claim. 
\end{proof}

Once we fix an isotopy between a homeomorphism $f: M \to M$ of a Riemannian manifold and the identity each periodic orbit is associated to a free homotopy class as follows:
\begin{definition}[Free homotopy class of a periodic orbit]
Let $M$ be a Riemannian manifold and $f:[0,1]\times M \to M$ be an isotopy with $f_0$ equal to the identity mapping of $M$.  If $x \in M$ is a periodic point for $f_1$ with minimal period $p$ then the free homotopy class $C(x)$ of $x$ is defined as the free homotopy class of the following closed curve:
\[\alpha:[0,p] \to M\]
\[\alpha(t) = f_{t-k}\circ f_1^k\text{ if } k \le t \le k+1\text{ where }k \in \mathbb{Z}\]
\end{definition}

Notice that if $x$ is a periodic orbit of period $n$ for a homeomorphism $f$ then $\mu = \frac{1}{n}\sum_{k = 0}^{n-1}\delta_{f^k(x)}$ is an ergodic invariant measure (where $\delta_p$ denotes the Dirac delta at a point $p$).  Any such measure will trivially satisfy the integrability hypothesis of Theorem \ref{maintheorem}.  Hence all periodic orbits have rotation vectors associated to them.  

The theorem below shows the relationship between the norm of the rotation vector of a periodic orbit and the length of its free homotopy class.

\begin{theorem}\label{periodicescaperate}
Let $\widetilde{M}$ be a Hadamard manifold and $\pi: \widetilde{M} \to M$ a Riemannian covering.  Suppose $f: M \to M$ is isotopic to the identity and $x \in M$ is a periodic point for $f$ of minimal period $p$.  Given $g:[0,1]\times M \to M$ an isotopy between the identity and $f$, if $F = G_1$ where $G:[0,1]\times \widetilde{M} \to \widetilde{M}$ is the unique lift of $g$ to $\widetilde{M}$ starting at the identity then $\|v_F(x)\| = l(C(x))/p$. 
\end{theorem}

\begin{proof}
To begin we observe that $F$ commutes with all covering transformation.  To see this let $\gamma$ be a covering transformation and notice that for each $t$ the equation $\rho_t = \gamma^{-1} \circ G_t^{-1}\circ \gamma \circ G_t$ defines a covering transformation.  Since for all $y \in \widetilde{M}$ the curve $t \mapsto \rho_t(y)$ is continuous and starts at $y$ it must be constant.  For $t = 1$ this gives $\gamma\circ F = F \circ \gamma$ as claimed.

Next fix a lift $\tilde{x}$ of $x$ and let $\rho: \widetilde{M} \to \widetilde{M}$ be the covering transformation such that $\rho(\tilde{x}) = F^p(\tilde{x})$.  Since $F$ commutes with $\rho$ it holds that $F^{np}(\tilde{x}) = \rho^n(\tilde{x})$ so that by lemma \ref{vectorsvrsescorts} one has:
\[p\|v_F(x)\| = \lim_{n \to +\infty}\frac{d(\tilde{x},\rho^n(\tilde{x}))}{n}\]

For any $y,z \in \widetilde{M}$ one has:
\[\lim_{n}\frac{d(y,\rho^n(y))}{n} \le \lim_{n}\frac{d(y,x) + d(x,\rho^n(x)) + d(\rho^n(x),\rho^n(y))}{n} = \lim_{n}\frac{d(x,\rho^n(x))}{n}\]

Hence $p\|v_F(x)\|$ equals the rate of escape of all points in $\widetilde{M}$ under iteration of $\rho$.

For each $n \ge 1$ we define:
\[R_n = \inf\{d(y,\rho^n(y)): y \in \widetilde{M}\}\]

We will prove that $p\|v_F(x)\| = R_1$ and by lemma \ref{lengthofaclass} the theorem follows.

First we establish $p\|v_F(x)\| \le R_1$ as follows:  For each $\epsilon > 0$ choose $y \in \widetilde{M}$ with $d(y,\rho(y)) \le R_1 + \epsilon$.  By the triangle inequality one has that
\[p\|v_F(x)\| = \lim_{n}\frac{d(y,\rho^n(y))}{n} \le R_1 + \epsilon\]
and by letting $\epsilon$ go to zero one establishes the desired inequality.

On the other hand, if the rate of escape of some point $y \in \widetilde{M}$ was less than $R_1$ then for all $n$ large enough one would have $d(y,\rho^{n}(y)) < nR_1$.  In particular this would imply $R_n < nR_1$ for all $n$ large enough.  Hence to establish the equality $R_1 = p\|v_F(x)\|$ it suffices to show that $R_n = n R_1$ whenever $n$ is a power of two\begin{footnote}{Here any unbounded sequence would do. In fact, once $R_1$ is characterized as the rate of escape of all points, one can deduce that $R_n = nR_1$ for all $n$.}\end{footnote}.

We will show $R_2 = 2R_1$ and the general claim follows from applying the same argument to iterates of $\rho$.

First consider a sequence $y_n$ in $\widetilde{M}$ with $d(y_n,\rho(y_n)) \to R_1$, it holds that:
\[R_2 \le d(y_n,\rho^2(y_n)) \le d(y_n,\rho(y_n)) + d(\rho(y_n),\rho^2(y_n))\]
which establishes (by taking limit of the right hand side) the inequality $R_2 \le 2R_1$.

Next consider a sequence $y_n$ in $\widetilde{M}$ such that $d(y_n,\rho^2(y_n)) \to R_2$ when $n \to +\infty$.  If $z_n$ is the midpoint of the geodesic segment $[y_n,\rho(y_n)]$ then, because $\rho$ is an isometry, $\rho(z_n)$ is the midpoint of $[\rho(y_n),\rho^2(y_n)]$.   Hence by lemma \ref{convexity} applied to the geodesic segments $[\rho(y_n),y_n]$ and $[\rho(y_n),\rho^2(y_n)]$ respectively we obtain:
\[R_1 \le d(z_n,\rho(z_n)) \le \frac{d(y_n,\rho^2(y_n))}{2} \to \frac{R_2}{2} \text{ when }n \to +\infty\]

From which the result follows.
\end{proof}

\section{Past orbits}
\label{sectionpast}
\subsection{Past and future rotation vectors}

We will begin this section with a lemma that shows in particular that the set of points $x \in M$ for which $v_F(x)$ exists is $f$-invariant, and that the norm $\|v_F(x)\|$ is constant along each orbit.
 
\begin{lemma}\label{orbitvectors}
Let $H$ be a Hadamard manifold and $\{x_n\}_{n \ge 0}, \{y_n\}_{n \ge 0} \subset H$ be such that:
\[d(x_n,y_n) = o(n)\text{ when }n\to +\infty\]

If there exists $v \in T_{x_0}H$ such that
\[d(x_n,\exp(nv)) = o(n)\text{ when }n \to +\infty\]
then there exists a unique vector $w \in T_{y_0}H$ such that:
\[d(y_n,\exp(nw)) = o(n)\text{ when }n \to +\infty\]

In such a case $\|w\| = \|v\|$ and if $\|v\| \neq 0$ then $v/\|v\|$ is asymptotic to $w/\|w\|$ (see Appendix). 
\end{lemma}
\begin{proof}
By \cite[Remark 4, p.48]{MR0336648} there is a unique vector $w \in T_{y_0}H$ such that:
\[d(\exp(tv),\exp(tw)) = O(1)\text{ when }t \to +\infty\]

By lemma \ref{convexity} (see also the discussion following definition \ref{asymptotictangentvectors}) $w$ is also the unique vector in $T_{y_0}H$ satisfying:
\[d(\exp(tv),\exp(tw)) = o(t)\text{ when }t \to +\infty\]

Notice that:
\[t\|v\| = d(x_0,\exp(tv)) \le d(x_0,y_0) + d(y_0,\exp(tw)) + d(\exp(tw),\exp(tv)) = t\|w\| + o(t)\]
which implies that $\|v\| \le \|w\|$ and hence (by symmetry) that $\|v\| = \|w\|$ as required.

Assuming $\|v\| = \|w\| \neq 0$ it follows from the bound on $d(\exp(tv),\exp(tw))$ that $v/\|v\|$ and $w/\|w\|$ are asymptotic.

To conclude we estimate $d(y_n,\exp(nw))$ as follows: 
\[d(y_n,\exp(nw)) \le d(y_n,x_n) + d(x_n,\exp(nv)) + d(\exp(nv),\exp(nw)) = o(n)\]
\end{proof}

\begin{corollary}\label{orbitvectors2}
Suppose $M,\widetilde{M},f$ and $F$ satisfy the hypothesis of Theorem \ref{maintheorem} and $x \in M$ has a rotation vector $v_F(x)$.  Then $v_F(f^k(x))$ exists and satisfies $\|v_F(f^k(x))\| = \|v_F(x)\|$ for all $k \in \mathbb{Z}$.
\end{corollary}
\begin{proof}
Let $\tilde{x}$ be a lift of $x$.  By definition of $v_F(x)$ the lift $\tilde{\alpha}$ of $t \mapsto \exp(tv_F(x))$ starting at $\tilde{x}$ satisfies:
\[d(\tilde{\alpha}(n),F^n\tilde{x}) = o(n)\text{ when }n \to +\infty\]

As an immediate consequence for any $k \in \mathbb{Z}$ one has:
\[d(F^{n}\tilde{x},F^{n+k}\tilde{x}) = o(n)\text{ when }n \to +\infty\]

The claim now follows by lemma \ref{orbitvectors} and the fact (observed in section \ref{statement}) that the existence of $v_F(f^k(x))$ is a property of the $F$-orbit of any single lift of $f^k(x)$.
\end{proof}

Since we are studying invertible dynamical systems it is natural to ask what the relationship is between rotation vectors for the system $f$ and its inverse $f^{-1}$.  In this direction we will prove the following theorem.

\begin{theorem}\label{past}
Let $M,\widetilde{M},f,F$ and $\mu$ satisfy the hypothesis of Theorem \ref{maintheorem}.  Then for almost every $x \in M$ it holds that $\|v_F(x)\| = \|v_{F^{-1}}(x)\|$ and either $v_F(x) = v_{F^{-1}}(x) = 0$ or $v_F(x) \neq v_{F^{-1}}(x)$.
\end{theorem}
\begin{proof}
Let $x \mapsto \tilde{x}$ be a measurable function such that $\tilde{x}$ projects to $x$ for all $x \in M$.  Consider for each $x \in M$ the Busemann function on $\widetilde{M}$ associated to the lift $\tilde{\alpha}$ of the geodesic $t \mapsto \exp(tv_F(x))$ starting at $\tilde{x}$:
\[B_x(p,q) = \lim_{t \to +\infty}d(p,\tilde{\alpha}(t)) - d(\tilde{\alpha}(t),q)\]

Notice that if $\tilde{\alpha}$ is non-constant then $B_x$ coincides with the Busemann function associated to the vector $\tilde{\alpha}'(0)/\|\tilde{\alpha}'(0)\|$ as defined in the Appendix.  The function $B_x$ satisfies $B_x(p,q) + B_x(q,r) = B_x(p,r)$ for all $p,q,r \in \widetilde{M}$.  Also, one has the inequality $|B_x(p,q)| \le d(p,q)$.

With this notation we define a measurable function $g:M \to \mathbb{R}$ by the formula $g(x) = B_x(\tilde{x},F\tilde{x})$.  Since $|g(x)| \le d(\tilde{x},F\tilde{x})$ the function $g$ is integrable.  

Also, if $\rho$ is a covering transformation then, because $F$ commutes with $\rho$, it follows that:
\[g(x) = \lim_{t \to +\infty}d(\rho(\tilde{x}),\rho\circ\tilde{\alpha}(t)) - d(\rho\circ\tilde{\alpha}(t), F(\rho(\tilde{x})))\]
so that the value of $g(x)$ is independent of the chosen lift of $x$.

We will now show that $\|v_F(x)\| \le \|v_{F^{-1}}(x)\|$ for almost all $x$.  If $v_F(x) = 0$ this is satisfied so we may assume $v_F(x) \neq 0$ in what follows.  

By lemma \ref{orbitvectors} and corollary \ref{orbitvectors2} it follows for all $k \in \mathbb{Z}$ that the lift of $t \mapsto \exp(tv_F(f^k(x))$ starting at $F^k(\tilde{x})$ is asymptotic to $\tilde{\alpha}$.  By lemma \ref{asymptoticbusemannfunctions} and the fact that $g(f^k(x))$ may be calculated at any lift of $f^k(x)$ this implies that:
\[g(f^k(x)) = B_{f^k(x)}(\widetilde{f^k(x)},F\widetilde{f^k(x)}) = B_x(F^k\tilde{x},F^{k+1}\tilde{x})\text{ for all }k\in \mathbb{Z}\]

Consider now the Birkhoff averages of $g$.  One has the following equality:
\[\frac{1}{n}\sum_{k = 0}^{n-1}g(f^k(x)) = \frac{1}{n}\sum_{k = 0}^{n-1}B_x(F^k\tilde{x},F^{k+1}\tilde{x}) = \frac{B_x(\tilde{x},F^n\tilde{x})}{n}\]

By definition of $v_F(x)$ one has that
\[d(\tilde{\alpha}(n),F^n\tilde{x}) = o(n)\]
from which one obtains the following:
\[n\|v_F(x)\| = B_x(\tilde{x},\tilde{\alpha}(n)) = B_x(\tilde{x},F^n\tilde{x}) + B_x(F^n\tilde{x},\tilde{\alpha}(n)) = B_x(\tilde{x},F^n\tilde{x}) + o(n)\]

Hence the limit $\tilde{g}$ of the Birkhoff averages of $g$ equals $\|v_F(x)\|$ almost surely.

But by Birkhoff's theorem one also obtains:
\begin{align*}
\|v_F(x)\| &= \lim_{n \to +\infty}\frac{1}{n}\sum_{k = 1}^{n}g(f^{-k}(x)) = \lim_{n \to +\infty}\frac{B_x(F^{-n}\tilde{x},\tilde{x})}{n}
\\&\le \lim_{n \to +\infty}\frac{d(F^{-n}\tilde{x},\tilde{x})}{n} = \|v_{F^{-1}}(x)\|
\end{align*}

This establishes that $\|v_F(x)\| \le \|v_{F^{-1}}(x)\|$.  By applying the same argument to $F^{-1}$ it follows that $\|v_F(x)\| = \|v_{F^{-1}}(x)\|$ almost everywhere.

Suppose now that for some $x$ we have $v_F(x) = v_{F^{-1}}(x)$.  Then $B_x$ is the Busemann function associated to the geodesics given by both rotation vectors.  In this case the above analysis and the property $B_x(p,q) = -B_x(q,p)$ give:
\[\|v_F(x)\| = \lim_{n \to +\infty}\frac{B_x(F^{-n}\tilde{x},\tilde{x})}{n} = -\lim_{n \to +\infty}\frac{B_x(\tilde{x},F^{-n}\tilde{x})}{n} = - \|v_{F^{-1}}(x)\|\]

Since both sides are non-negative this implies that $v_F(x) = v_{F^{-1}}(x) = 0$.
\end{proof}

On strong visibility manifolds (e.g. manifolds with Anosov geodesic flow, see Appendix) we obtain the following corollary.
\begin{corollary}\label{biinfinite}
Let $M,\widetilde{M},f,F$ and $\mu$ satisfy the hypothesis of Theorem \ref{maintheorem} and suppose $M$ is a strong visibility manifold.  Then for almost every $x \in M$ with $v_F(x) \neq 0$ it holds that for each lift $\tilde{x}$ of $x$ there exists a geodesic $\tilde{\alpha}:\mathbb{R} \to \widetilde{M}$ such that:
\[d(\tilde{\alpha}(n),F^n\tilde{x}) = o(n)\text{ when }n \to \pm\infty\]
Furthermore $\tilde{\alpha}$ is unique up to reparametrizations by time translation. 
\end{corollary}
\begin{proof}
Set $R = \|v_F(x)\|$ and let $v^+,v^- \in T_{\tilde{x}}\widetilde{M}$ project to $v_F(x)/R$ and $v_{F^{-1}}(x)/R$ respectively.  Consider the geodesic defined by:
\[\tilde{\alpha}(t) = \exp(tRv)\]
where $v = \pi_h(v^+,v^-)$ and $\pi_h$ is the projection along horospheres (Lemma \ref{projectionalonghorospheres}).
The uniqueness claim follows from the definition of strong visibility (see Appendix).
\end{proof}

We will call a geodesic satisfying the properties of the above corollary a bi-infinite geodesic escort for the sequence $\{F^n\tilde{x}\}_{n \in \mathbb{Z}}$.

\subsection{An example without a bi-infinite geodesic escort}

The limitations of preceding corollary are illustrated by the following example (see \cite{MR1729880} Remark 2.4).  Consider the metric $\mathrm{d}s^2 = (1+e^{-y})^2dx^2 + dy^2$ in $\mathbb{R}^2$.  Direct calculation shows that this metric has non-positive curvature.  Note that (compare with Theorem 2.1 in \cite{MR1639844}), a curve $\alpha = (x,y): \mathbb{R} \to \mathbb{R}^2$ is a geodesic if and only if the following Hamiltonian equations are satisfied:

\[\left\{\begin{array}{c}x' = (1+e^{-y})^{-2}p_1\\y' = p_2\\p_1' = 0\\p_2' = -(1+e^{-y})^{-3}e^{-y}p_1^2\end{array}\right.\]

In particular $p_1 =  (1+e^{-y})^2x'$ is constant along each geodesic.  Since $E = (1+e^{-y})^2x'^2 + y'^2 = (1+e^{-y})^{-2}p_1^2 + y'^2$ (i.e. the square of the norm of the velocity) is also constant along each geodesic one can classify the asymptotic behavior of all geodesics as follows:
\begin{itemize}
\item If $E > p_1^2$ then $y'^2$ is bounded away from $0$.  This implies that $y'$ has constant sign and is bounded away from $0$.
\item If $E = p_1^2$ then $y'^2 = (1-(1+e^{-y})^{-2})E$ (e.g. it is never $0$ unless the geodesic is constant).  From this one deduces that $y$ is a bijection of the real line with $y' \to 0$ when $y \to +\infty$ and $|y'| \to +\infty$ when $y \to -\infty$.  Depending on sign of $y'(0)$ the function $y$ is decreasing or increasing.
\item If $E < p_1^2$ then $y$ is bounded and $y''$ is bounded from above by a negative constant.  In particular $y \to -\infty$ and $|y'| \to +\infty$ when $t \to \pm \infty$.
\end{itemize}

Consider the sequence $x_n = (n,0) \in \mathbb{R}^2$ for $n \in \mathbb{Z}$.  It follows from comparison with the usual Euclidean metric that the rate of escape of this sequence greater then or equal to $1$.  One can show that, in fact, the rate of escape is exactly $1$ by calculating the length of the polygonal path through the points $(0,0),(0,\log(n)), (n,\log(n))$ and $(n,0)$ (which turns out to be $n + 2\log(n) + 1$).

Since $\{x_n\}$ is the orbit of an isometry it is aligned by Lemma \ref{alignmentoforbitsofsemicontractions}.  Therefore one obtains by Lemma \ref{geometriclemma} that it must have a geodesic escort $\alpha$ starting at $(0,0)$.  The considerations above ensure that this geodesic must be of the type with $E = p_1^2$.  The geodesic $\beta$ symmetric to $\alpha$ with respect to the $y$-axis escorts the sequence $x_n$ when $n \to -\infty$ (because symmetry with respect to the $y$-axis is an isometry).  However there is no bi-infinite geodesic $\gamma$ with $d(\gamma(n),x_n) = o(n)$ when $n \to \pm \infty$, as one can see by checking each of the three types of geodesics discussed above.

Consider now $\widetilde{M} = \mathbb{R}^2$ with the metric under discussion, $F$ given by $F(x,y) = (x+1,y)$ and $M = \widetilde{M}/\Gamma$ where $\Gamma$ is the group generated by any non-trivial translation $\gamma(x,y) = (x+t,y)$.  One can take $f$ to be the projection of $F$ to $M$ and $\mu$ to be the projection of the uniform probability measure on $[0,t]\times\{0\}$.  With these definitions one is in hypothesis of Theorems \ref{maintheorem} and \ref{past}.  

Let $p$ be the projection to $M$ of $(0,0)$, the rotation vector $v_F(p)$ is the projection of $\alpha'(0)$ to $TM$ and the rotation vector $v_{F^{-1}}(p)$ is the projection of $\beta'(0)$.  Hence the two vectors are different (since $\alpha'(0)$ has a non-null $x$-component).  However one cannot apply Corollary \ref{biinfinite} and in fact, as we have shown, the conclusions of this corollary do not hold in this case.

This example also shows that the distance to the geodesic escort given by a rotation vector can be non-bounded.

\section{Rotation vectors for flows}
\label{sectionflows}
We will now develop the concept of rotation vectors for continuous flows.  Let us begin with the definition.

\begin{definition}[Rotation vector for a flow]
Let $M$ be a complete Riemannian manifold, $\widetilde{M}$ its universal covering space, $f: \mathbb{R}\times M \to M$ a continuous flow, and $F: \mathbb{R}\times\widetilde{M} \to \widetilde{M}$ its lift.

A rotation vector $v(x)$ of a point $x \in M$ is a vector in the tangent space $T_xM$ such that for any lift $\tilde{x}$ of $x$ the following holds:
\[d(\tilde{\alpha}(t),F^t\tilde{x}) = o(t)\text{ when }t\to +\infty\]
where $\tilde{\alpha}$ is the lift starting at $\tilde{x}$ of the geodesic $\alpha:[0,+\infty) \to M$ defined by $\alpha(t) = \exp_x(tv(x))$.
\end{definition}

An important difference between the case of flows and that of homeomorphisms is that for flows there is a unique lift which automatically commutes with all covering transformations.  Also the geodesic escort given by a rotation vector satisfies an asymptotic condition on all $t \in [0,+\infty)$ (as opposed to only integer values).  In order to deduce an existence theorem for flows from the corresponding theorem for homeomorphisms, it is necessary to control the variation of a trajectory between integer times (compare with the discussion in \cite{MR0356192} section 1.4).

\begin{theorem}[Existence of rotation vectors for flows]\label{flows}
Let $M$ be a complete connected Riemannian manifold with non-positive sectional curvature and let $\widetilde{M}$ be its universal covering space.  For each $x \in M$ let $\tilde{x} \in \widetilde{M}$ denote an arbitrary lift of $x$.

Suppose that $f: \mathbb{R}\times M \to M$ is a continuous flow with lift $F:\mathbb{R}\times \widetilde{M} \to \widetilde{M}$, and that $\mu$ is an $f$-invariant Borel probability measure satisfying the condition:
\[\int_M \sup_{0 \le s \le t \le 1}d(F^s\tilde{x},F^t\tilde{x})\mathrm{d}\mu(x) < +\infty\]

Then for $\mu$-almost every $x \in M$ there exists a unique rotation vector $v(x) \in T_xM$.
\end{theorem}
\begin{proof}
By Theorem \ref{maintheorem} for almost every $x \in M$ there is a unique vector $v(x) \in T_xM$ such that:
\[d(\tilde{\alpha}(n),F^n\tilde{x}) = o(n) \text{ when } n \to +\infty, n \in \mathbb{Z}\]
where $\tilde{\alpha}$ is the lift starting at $\tilde{x}$ of $t \mapsto \exp(tv(x))$.

The function $g:M \to [0,+\infty)$ given by $g(x) = \sup_{0 \le s \le t \le 1}d(F^s\tilde{x},F^t\tilde{x})$ is well 
defined and continuous. In order to prove that
\[d(\tilde{\alpha}(t),F^t\tilde{x}) = o(t) \text{ when } t \to +\infty\]
if suffices to show that for almost all $x \in M$ it holds that $g(f^nx) = o(n)$ when $n \to +\infty$.  This follows from Birkhoff's ergodic theorem which is applicable because $g$ is integrable.
\end{proof}

\section{A Semi-conjugacy result for Flows on Fiber Bundles}
\label{sectionbundles}
In this section we will show how the existence of non-null rotation vectors provides a measurable semi-conjugacy between a given flow and the geodesic flow of a manifold.  This generalizes the semi-conjugacy statement of Theorem 4.1 in \cite{MR1802657} by replacing the restriction of constant negative curvature by that of strong visibility (see Appendix).

We will begin by defining rotation vectors through a map between two manifolds:
\begin{definition}
Let $\widetilde{M}$ be a Hadamard manifold, $\pi:\widetilde{M} \to M$ a Riemannian covering, $f: \mathbb{R}\times N \to N$ a continuous flow on a manifold $N$ and $h:N \to M$ a continuous map.

A rotation vector $v_h(x)$ of a point $x \in N$ is a vector in the tangent space $T_{h(x)}M$ such that for any lift $\tilde{p}$ of $p = h(x)$ the following holds:
\[d(\tilde{\alpha}(t),\tilde{\beta}(t)) = o(t)\text{ when }t\to +\infty\]
where $\tilde{\alpha}$ is the lift starting at $\tilde{p}$ of the geodesic $\alpha:[0,+\infty) \to M$ defined by $\alpha(t) = \exp_x(tv(x))$ and $\tilde{\beta}$ is the lift starting at the same point of the curve $\beta(t) = h(f^t(x))$.
\end{definition}

This definition applies in particular to the geodesic flow on $TM$ by taking the map $h:TM \to M$ to be the natural projection.  In this case the rotation vector associated to a point $v \in TM$ is simply $v$.

When $f:\mathbb{R}\times TM \to TM$ is the geodesic flow of a different metric then the one we are considering on $M$ rotation vectors provide a comparison between geodesics of the two metrics.  In general, if $p:E \to M$ is a fiber bundle projection and $f:\mathbb{R}\times E \to E$ is a continuous flow rotation vectors through $p$ provide a comparison between the trajectories of $f$ on the base manifold, and the geodesic flow of this manifold as in \cite{MR1802657}.

As another example suppose $f: \mathbb{R} \times S(N) \to S(N)$ is the geodesic flow on the unit tangent bundle $S(N)$ of Riemannian manifold $N$ and $h: S(N) \to M$ is the composition of the natural projection from $S(N)$ to $N$ and a map $h_1: N \to M$.  In this case the integral of the norm of the associated rotation vectors (i.e. the rate of escape) is equal to the `intersection' $i_{h_1}(g_N,g_M)$ of the geodesic flows $g_N$ and $g_M$ on $N$ and $M$ respectively introduced in \cite{MR1082022}.

The following existence theorem is a corollary of Theorem \ref{alignmentofrandomsequences} by arguments very close to the proof of Theorem \ref{maintheorem}.
\begin{theorem}
\label{bundleexistence}
Suppose $f:\mathbb{R}\times N \to N$ is a $C^1$ flow on a manifold $N$ preserving a Borel probability measure $\mu$ with compact support.  If $M$ is a connected complete Riemannian manifold with non-positive sectional curvature and $h: N \to M$ is a $C^1$ map then for $\mu$-almost every $x \in N$ there exists a unique rotation vector $v_h(x)$.
\end{theorem}
\begin{proof}
Let $\pi: \widetilde{M} \to M$ be the universal Riemannian covering of $M$ and let $L:M \to \widetilde{M}$ be measurable such that $\pi(L(x)) = x$ for all $x \in M$.

Define $\phi: N \to \widetilde{M}^{\mathbb{N}}$ so that, for all $x \in N$ and $n \in \mathbb{N}$, if $\alpha_x:[0,+\infty) \to \widetilde{M}$ is the lift starting at $L(h(x))$ of the curve $t \mapsto h(f^t x)$ then $\phi(x)_n = \alpha_x(n)$.

Notice that if $X: N \to TN$ is the vector field generating the flow $f$ and $C = \max_{supp(\mu)} \|Dh X\|$, where the maximum is taken on the support of $\mu$, then:
\[\int_N d(\phi(x)_0,\phi(x)_1)\mathrm{d}\mu(x) \le \int_N\int_0^1 \|Dh X(f^s x)\| \mathrm{d}s\mathrm{d}\mu(x) \le C\]

Also one has that
\[d(\phi(x)_{m+1},\phi(x)_{n+1}) = d(\alpha_x(m+1),\alpha_x(n+1)) = d(\alpha_{f^1(x)}(m),\alpha_{f^1(x)}(n))\]
where the second equality is obtained by applying the covering transformation taking $\alpha_x:[1,+\infty) \to \widetilde{M}$ to $\alpha_{f(x)}:[0,+\infty)\to \widetilde{M}$ (which are two lifts of $t \mapsto f^t(f^1(x))$).

Hence $\phi$ satisfies the hypothesis of Theorem \ref{alignmentofrandomsequences} and in consequence for $\mu$ almost every $x \in N$, either because the rate of escape is zero or, if it is positive, by Lemma \ref{geometriclemma}, there is a unique geodesic $\beta_x:[0,+\infty) \to \widetilde{M}$ starting at $L(h(x))$ and escorting the sequence $\alpha_x(n)$.  

For such a point $x \in N$ let $R$ be the rate of escape of $\{\alpha_x(n)\}_{n \ge 0}$.  To conclude that the projection of $R\beta'(0)$ to $TM$ is a rotation vector for $x$ through $h$ it suffices to note that if $n \le t \le n+1$ where $n \in \mathbb{N}$ one has:
\[d(\alpha_x(t),\alpha_x(n)) \le C\]

The uniqueness of $v_h(x)$ for the points where it is defined follows directly from the uniqueness statement for $\beta_x$. 
\end{proof}

We define the rotation vectors to the past as those for the time-reversed flow $g^tx = f^{-t}x$.  In the same spirit as Section \ref{sectionpast} we obtain the following:
\begin{theorem}\label{bundlepast}
With the hypothesis and notation of Theorem \ref{bundleexistence} and the additional assumption that $\mu$ is ergodic, let $v_h^-(x)$ denote the rotation vector to the past of a point $x \in N$ if it exists.  For almost every $x \in N$ it holds that $\|v_h(x)\| = \|v_h^-(x)\|$ and, either $v_h(x) = v_h^-(x) = 0$ or $v_h^-(x) \neq v_h(x)$.
\end{theorem}
\begin{proof}
Let $\pi: \widetilde{M} \to M$ be the universal Riemannian covering of $M$.

Notice that, because $\mu$ is ergodic, there exist constants $R^+,R^-$ such that $\|v_h(x)\| = R^+$ and $\|v_h^-(x)\| = R^-$ for almost every $x \in N$.  If $R^+ = R^- = 0$ then there is nothing to prove.  Otherwise by a linear time change (of the form $t \mapsto \pm Ct$) we may assume that $R^+ = 1$.

For each $x \in N$ with $\|v_h(x)\| = 1$ let $\alpha_x: \mathbb{R} \to \widetilde{M}$ be a lift of the curve $t \mapsto h(f^t x)$.  Also, let $v(x) \in T_{\alpha_x(0)}\widetilde{M}$ project to $v_h(x)$ and define $g:N \to \mathbb{R}$ as follows (see Appendix for definition of $B_v$):
\[g(x) = B_{v(x)}(\alpha_x(0),\alpha_x(1))\]

By noting that $B_{D\rho v}(\rho(p),\rho(q)) = B_v(p,q)$ for any $p,q \in \widetilde{M}$, $v \in S(\widetilde{M})$ and any isometry $\rho: \widetilde{M} \to \widetilde{M}$ it follows that $g$ doesn't depend on the chosen lift $\alpha_x$ and in particular is measurable (by choosing $x \mapsto \alpha_x(0)$ measurable).  This also gives the following equalities:
\[\sum_{k = 0}^{n-1}g(f^kx) = \sum_{k = 0}^{n-1}B_{v(x)}(\alpha_x(k),\alpha_x(k+1)) = B_{v(x)}(\alpha_x(0),\alpha_x(n))\]
\[\sum_{k = 1}^{n}g(f^{-k}x) = \sum_{k = 1}^{n}B_{v(x)}(\alpha_x(-k),\alpha_x(-k+1)) = B_{v(x)}(\alpha_x(-n),\alpha_x(0))\]

Because $v(x)$ projects to $v_h(x)$ which has norm $1$ it follows that $B_{v(x)}(\alpha_x(0),\exp(nv(x))) = n$.  Also,
\[B_{v(x)}(\alpha_x(0),\alpha_x(n)) = B_{v(x)}(\alpha_x(0),\exp(nv(x)))) + B_{v(x)}(\exp(nv(x)),\alpha_x(n))\]

From this and the fact that 
\[|B_{v(x)}(\exp(nv(x)),\alpha_x(n))| \le d(\exp(nv(x),\alpha_x(n)) = o(n)\]
one obtains that
\begin{align}\label{equationnormofv}
\lim_{n \to +\infty}\frac{1}{n}B_{v(x)}(\alpha_x(0),\alpha_x(n)) = \lim_{n \to +\infty}\frac{1}{n}\sum_{k = 0}^{n-1}g(f^k x) = 1 = \|v_h(x)\|
\end{align}
for all $x \in N$ with $\|v_h(x)\| = 1$. 

Furthermore one has $|g(x)| \le d(\alpha_x(0),\alpha_x(1))$ which is bounded on the support of $\mu$, hence by Birkhoff's ergodic theorem and the ergodicity of $\mu$:
\begin{align}\label{equationbirkhoffpast}
1 = \lim_{n \to +\infty}\frac{1}{n}\sum_{k = 1}^ng(f^{-k}(x)) = \lim_{n \to +\infty}\frac{1}{n}B_{v(x)}(\alpha_x(-n),\alpha_x(0))
\end{align}
for almost every $x \in N$.

Since $|B_{v(x)}(\alpha_x(-n),\alpha_x(0))| \le d(\alpha_x(-n),\alpha_x(0))$ for all $n$ this implies that $\|v_h^-(x)\| \ge 1 = \|v_h(x)\|$ for almost every $x \in N$.  However by applying the argument to a time reversing reparametrization of $f$ it follows that $\|v_h(x)\| = \|v_h^-(x)\| = 1$ for almost every $x \in N$.

To show that $v_h(x) \neq v_h^-(x)$ for almost every $x \in N$, we proceed by contradiction.  If $v_h(x) = v_h^-(x)$ with positive probability then for some $x \in N$ we would have:
\begin{align*}
1 &= \|v_h^-(x)\| = \lim_{n \to +\infty}\frac{1}{n}B_{v(x)}(\alpha_x(0),\alpha_x(-n))
\\& = -\lim_{n \to +\infty}\frac{1}{n}B_{v(x)}(\alpha_x(-n),\alpha_x(0)) = -1
\end{align*}
where the second equality is given by equation \ref{equationnormofv} applied to the time-reversal of $f$ and last equality follows from equation \ref{equationbirkhoffpast}.
\end{proof}

In order to state our semi-conjugacy result let us fix the following definitions.

\begin{definition}[Cocyle]
Let $f: \mathbb{R}\times N \to N$ be a measurable flow on a measurable space $(N,\mathcal{B}_N)$.  A cocycle over $f$ is a measurable function $a: N \times \mathbb{R} \to \mathbb{R}$ such that:
\[a(x,s+t) = a(x,s) + a(f^sx,t)\text{ for all }s,t \in \mathbb{R},x \in N\]
The cocycle $a$ is said to be invertible if $a(x,\cdot): \mathbb{R} \to \mathbb{R}$ is an increasing homeomorphism for all $x$.
\end{definition}

\begin{definition}[Measurable semi-conjugacy]
A measure preserving flow $f:\mathbb{R}\times N \to N$ on a probability space $(N,\mathcal{B}_N,\mu)$ is said to be semi-conjugate to a measurable flow $g: \mathbb{R}\times M \to M$ on a measurable space $(M,\mathcal{B}_M)$ if there exists an invertible cocycle $a:N \times \mathbb{R} \to \mathbb{R}$ and a measurable function $\phi:N \to M$ such that:
\[\phi(f^t x) = g^{a(x,t)}\phi(x)\text{ for all }t \in \mathbb{R},x \in N\]
\end{definition}

The last definition is quite lax, e.g. all flows are semi-conjugate to the constant flow on the space containing a single point.  However, we will prove a semi-conjugacy result with geometric content since the conjugation will associate to each point a geodesic escort of the image of its orbit through a given map.

\begin{theorem}[Semi-conjugacy]\label{semiconjugacytheorem}
With the hypothesis and notation of Theorem \ref{bundlepast} suppose that $M$ is a strong visibility manifold and that $\|v_h(x)\| = R > 0$ for almost every $x \in N$.

Then there exists an invariant Borel set $G \subset N$ with $\mu(G) = 1$ and such that $f$ restricted to $G$ is measurably semi-conjugate to the geodesic flow $g:\mathbb{R}\times TM \to TM$ of $M$.  Furthermore, the semi-conjugation $\phi: G \to TM$ can be chosen so that for all $x \in G$ there exist $\tilde{\alpha},\tilde{\beta}:\mathbb{R} \to \widetilde{M}$ which are lifts of the curves $t \mapsto \exp(t\phi(x))$ and $t \mapsto h(f^t x)$ respectively, and have the following property:
\[d(\tilde{\alpha(t)},\tilde{\beta}(t)) = o(t) \text{ when }t \to \pm\infty\]
\end{theorem}
\begin{proof}
By linear time change of the flow $f$ we may suppose $R = 1$

Let $G \subset N$ be the set of points possessing distinct rotation vectors with norm $1$ to the past and the future.  This set is invariant, and has full measure by Theorems \ref{bundleexistence} and \ref{bundlepast}.

We first note that the functions $v_h, v_h^-: G \to TM$ are measurable.  In fact if one defines the curve $t \mapsto v_h^t(x)$ to be the lift of $t \mapsto h(f^t x)$ to $T_{h(x)}M$ via the exponential map, then $v_h^t$ is a continuous function of $x \in G$ and for all $x \in G$ one has:
\[v_h(x) = \lim_{t \to +\infty}\frac{1}{t}v_h^t(x)\]
\[v_h^-(x) = \lim_{t \to +\infty}\frac{1}{t}v_h^{-t}(x)\]

Now fix $x \in G$, consider any lift $y \in \widetilde{M}$ of $h(x)$ and define $\tilde{\beta}: \mathbb{R} \to \widetilde{M}$ as the lift of $t \mapsto h(f^t x)$ starting at $y$.  

Let $v^+,v^- \in T_y\widetilde{M}$ project to $v_h(x)$ and $v_h^-(x)$ respectively and let $v = \pi_h(v^+,v^-)$ be the projection along horospheres of $v^+,v^-$ (see Lemma \ref{projectionalonghorospheres}).    Defining $\tilde{\alpha}(t) = \exp(tv)$ it holds that:
\[d(\tilde{\alpha}(t),\tilde{\beta}(t)) = o(t) \text{ when } t \to \pm \infty\]

Also, if $\gamma$ is a covering transformation then one can see that $D\gamma v^+$ remains asymptotic (see Appendix) to $D\gamma v$ and similarly for $D\gamma v^-$ and $-D\gamma v$, also the base points remain on the same horosphere with respect to $B_{D\gamma v}$ and therefore it must be that $D\gamma v = \pi_h(D\gamma v^+,D\gamma v^-)$.  This shows that the projection of $v$ to $TM$ doesn't depend on the chosen lift $y$.  Hence the equation:
\[\phi_1(x) = D\pi(v)\]
defines a measurable function $\phi_1: G \to TM$.

For any $t \in \mathbb{R}$ let $v^+(t)$ and $v^-(t)$ denote the lifts of $v_h(f^t x)$ and $v_h^-(f^t x)$ to $T_{\tilde{\beta}(t)}\widetilde{M}$.  Since $v^+(t)$ is asymptotic to $v^+$ for all $t$, Lemma \ref{asymptoticquotienttopology} implies that $v^+(t)$ is a continuous function of $t$.  Similarly $v^-(t)$ is a continuous function of $t$.  This implies that there exists a unique $b(x,t) \in \mathbb{R}$, which is continuous with respect to $t$, such that:
\[\phi_1(f^t x) = g^{b(x,t)}\phi_1(x)\]

From the above equation the function $b: G \times \mathbb{R} \to \mathbb{R}$ is a measurable cocycle.  However the cocycle $b$ is not necessarily invertible.

To construct an invertible cocycle we notice that since $\tilde{\alpha}$ and $\tilde{\beta}$ are asymptotic one has:
\[\lim_{t \to +\infty}\frac{b(x,t)}{t} = 1\text{ for all }x \in G\]

And also if $C = \max_{supp(\mu)} \|Dh X\|$ one has:  
\[|b(x,t)| \le \int_0^t \|Dh X(f^s x)\|\mathrm{d}s \le C t\]
so that $b$ is uniformly Lipshitz with respect to $t$.

Under these conditions \cite[Lemma 1.4]{MR1802657} guarantees the existence of a measurable function $r: G \to (0,+\infty)$ such that the equation
\[a(x,t) = b(x,t) + r(f^t x) - r(x)\]
defines an invertible cocycle which satisfies the property:
\[\lim_{t \to +\infty}\frac{a(x,t)}{t} = 1\text{ for all }x \in G\]

Hence by defining $\phi(x) = g^{r(x)}\phi_1(x)$ the theorem follows.
\end{proof}

\section{Examples}
\label{sectionexamples}
\subsection{Homeomorphisms of hyperbolic surfaces}

Besides the two-dimensional torus the simplest manifolds to which our results are applicable are complete hyperbolic surfaces.  Self-homeomorphisms of such manifolds which are isotopic to the identity form a large family of dynamical systems which have been widely studied with other tools.  We will now discuss some implications of the existence of rotation vectors in this context and relate them to known results.

Consider the disk $\mathbb{D} = \{z \in \mathbb{C}: |z| < 1\}$ endowed with the Poincare metric.  We recall that geodesics of this metric are parametrizations of either Euclidean diameters of $\mathbb{D}$ or Euclidean circle arcs which are perpendicular to the boundary $\partial \mathbb{D}$.  In particular, we note that any geodesic ray has a well defined endpoint in $\partial\mathbb{D}$.  The only additional fact we will need about this metric is that the closed disk centered at a point $x \in \mathbb{D}$ with hyperbolic radius $R \ge 0$ (which we will denote by $D(x,R)$) is in fact an Euclidean closed disk centered at a point on the Euclidean segment $[0,x]$ (see the proof of \cite[Proposition 4.5]{MR2161463}).

Throughout this section let $M$ be a complete hyperbolic surface and $f: M \to M$ a homeomorphism which is isotopic to the identity.  Fix a Riemannian covering $\pi: \mathbb{D} \to M$ and let $F: \mathbb{D} \to \mathbb{D}$ be an admissible lift of $f$.  The following proposition summarizes the consequences of the existence of points with non-null rotation vectors.

\begin{proposition}
\label{hyperbolicsurface}
Let $x \in \mathbb{D}$ project to a point in $M$ possessing rotation vectors to the future and past which are non-null, distinct, and have equal norm.  Then there exist two distinct boundary points $x_-,x_+ \in \partial\mathbb{D}$ such that:
\[\lim_{n \to +\infty}F^n(x) = x_+\]
\[\lim_{n \to +\infty}F^{-n}(x) = x_-\]
\end{proposition}
\begin{proof}
Let $v_+,v_- \in T_x\mathbb{D}$ project to the rotation vectors to the past and future of $\pi(x)$ respectively,  and let $x_+,x_- \in \partial \mathbb{D}$ be the boundary endpoints of the geodesic rays with initial condition $v_+$ and $v_-$ respectively. One has that:
\[d(\exp(nv_+),F^n(x)) = o(n)\text{ when }n \to +\infty\]
\[d(\exp(nv_-),F^{-n}(x)) = o(n)\text{ when }n \to +\infty\]

In particular it follows that $d(0,F^n(x)) = \|v_+\|n + o(n)$ so that any limit point of the sequence $\{F^n(x)\}_{n \ge 0}$ will belong to the boundary $\partial \mathbb{D}$.  Define $x_n = \exp(nv_+)$ and notice that $R_n = d(x_n,0) = n\|v_+\| + O(1)$.  If follows that $F^n(x) \in D(x_n,R_n) = D_n$ for all but finitely many $n \ge 0$.  Since $D_n$ is an Euclidean disk centered on a point in the segment $[0,x_n]$ and $0$ belongs to the boundary of $D_n$ it follows that $D_n$ is contained in the closed Euclidean disk with diameter $[0,x_n/|x_n|]$.  Since $\lim_{n \to +\infty}x_n/|x_n| = x_+$ one obtains that $\lim_{n \to +\infty}F^n(x) = x_+$ as claimed.  

The same argument establishes the claim for for $\{F^{-n}(x)\}_{n \ge 0}$ and $x_-$.
\end{proof}

The conclusion of the above proposition was shown by Michael Handel to hold in some cases in which $x$ is the lift of a periodic orbit of $f$ as discussed in \cite[Section 3]{MR1325916} (compare with Theorems \ref{periodicescaperate} and \ref{past}).  The existence of points in the lift whose orbits have distinct limits on the boundary circle when $n \to \pm\infty$ can in certain circumstances imply the existence of fixed points of $f$ with index $1$ by application of Handel's fixed point theorem as in \cite[Proposition 3.4]{MR1325916} (see also \cite{MR1660349}).

In particular from the observations above it follows that if a point $x \in \mathbb{D}$ projects to a point with non-null and distinct rotation vectors to the past and future then its orbit cannot remain close to a horocycle since this would imply $x_+ = x_-$ contradicting Proposition \ref{hyperbolicsurface} (compare with the example at the beginning of section \ref{periodic}).

\subsection{Hyperbolic Magnetic Flow}

The differential equation 
\[\alpha''(t) = i\alpha'(t)\text{ for all }t\]
where $\alpha:\mathbb{R} \to \mathbb{R}^2$ and multiplication by $i$ denotes rotation by $90$ degrees in the counter-clockwise direction, determines the trajectory of a charged particle in a plane orthogonal to a constant magnetic field in Euclidean space.

It is easy to show that any such curve is periodic and in fact is a parametrization of an Euclidean circle of radius $\|\alpha'(0)\|$.

We will now consider a similar dynamical system in the hyperbolic plane, that is, we consider the differential equation:
\begin{equation}
\label{magneticequation}
\frac{D}{\mathrm{d}t}\alpha'(t) = i \alpha'(t)
\end{equation}
where $D/\mathrm{d}t$ denotes covariant derivation with respect to the canonical connection of the hyperbolic metric.  This differential equation defines a flow on the tangent bundle of the hyperbolic plane and also on the tangent bundle of any oriented hyperbolic surface (this is the so called `magnetic flow' for the volume $2$-form, e.g. see \cite{MR2343689}).

Fixing a non-constant trajectory and denoting the norm of its velocity (which remains constant) by $v$ we will prove the following (see figure \ref{magneticfigure}
):
\begin{itemize}
\item If $v < 1$ then the trajectory is periodic and in fact is a hyperbolic circle of radius $r$ where $\tanh(r) = v$.
\item If $v = 1$ then the trajectory is a horosphere.
\item If $v > 1$ then the trajectory has rate of escape $R = \sqrt{v^2 - 1}$ and describes a curve which is at a constant distance $r > 0$ from a geodesic.  The distance $r$ is given by the equation $\tanh(r) = 1/v$.
\end{itemize}

\begin{figure}
\begin{center}
\label{magneticfigure}
\includegraphics[height=40mm,width=40mm]{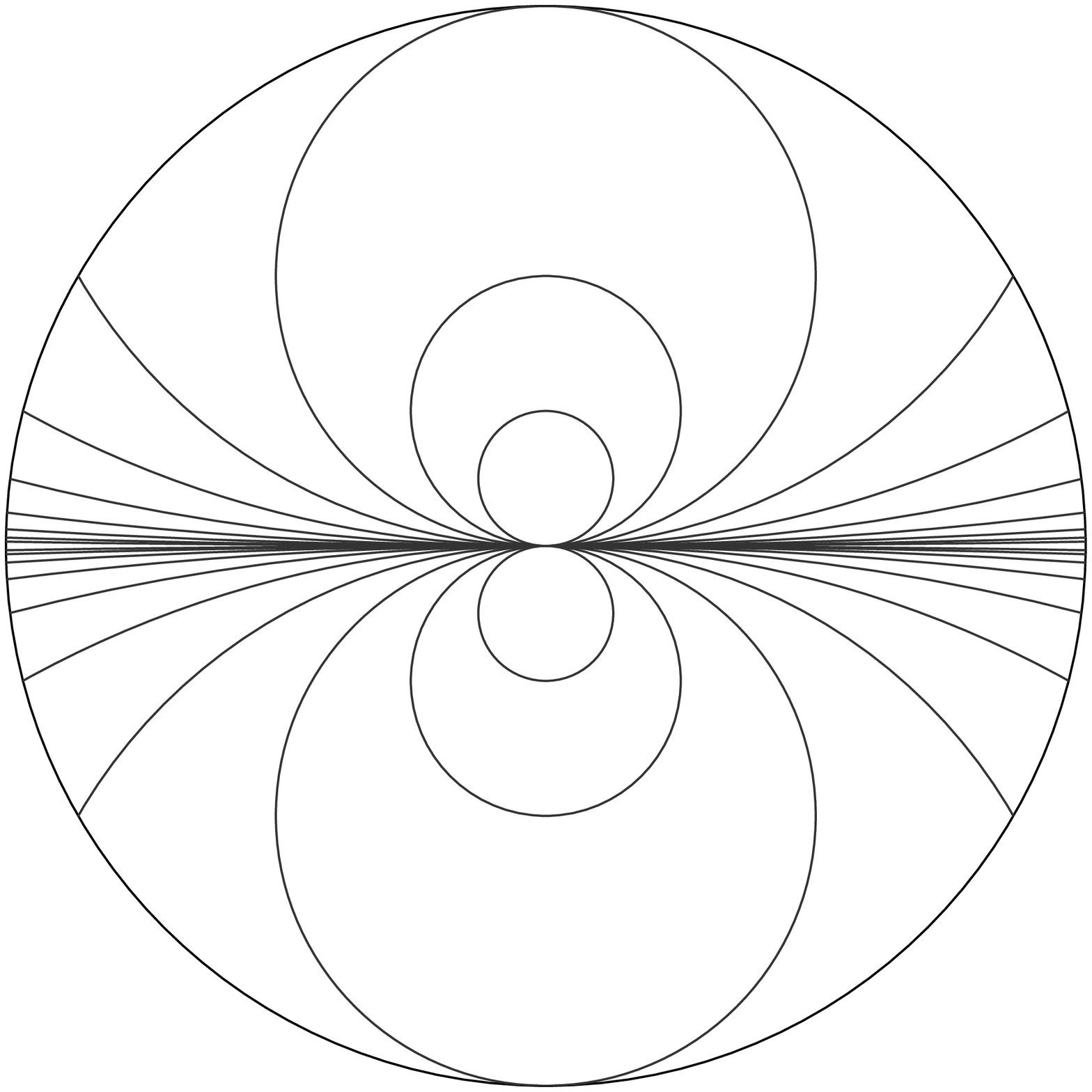}
\end{center}
\begin{caption}{Trajectories parting from $0$ in the Poincar\'e disk model of the hyperbolic plane are Euclidean circle arcs and are traversed in the counter-clockwise sense.}
\end{caption}
\end{figure}

In particular, if $M$ is a compact oriented hyperbolic surface and $f:\mathbb{R}\times TM \to TM$ is the magnetic flow described above then any ergodic invariant measure for $f$ which is supported on a set of vectors with constant norm $v > 1$ has non-null rotation vectors and satisfies the hypothesis of Theorem \ref{semiconjugacytheorem}.

Since the set of solutions to equation \ref{magneticequation} is invariant by isometries and the description above accounts for all possible values of the initial speed $v$ all one needs to do to establish the claims is calculate that the described curves are in fact solutions to the equation.  This verification amounts to calculating the covariant derivative of an explicit curve and is simple to do if an appropriate coordinate system (which we will now indicate) is used for each case.

For the case $v < 1$ the claim follows from \cite{MR1917223} but one can also verify it by calculating in polar coordinates.  Concretely consider the plane $P = \{(r,\theta) \in \mathbb{R}^2: r > 0\}$ with the metric $\mathrm{d}r^2 + \sinh(r)^2\mathrm{d}\theta^2$.  Let $\Gamma$ be the group generated by the isometry $(r,\theta) \mapsto (r,2\pi + \theta)$.  The Riemannian quotient manifold $P /\Gamma$ is isometric to the hyperbolic plane minus one point (this follows from the fact that the perimeter of the hyperbolic circle of radius $r$ is $2\pi \sinh(r)$).  Our claim amounts to the statement that for each $r > 0$ one has:
\[\frac{D}{\mathrm{d}t}\alpha'(t) = (-\tanh(r),0)\]
where
\[\alpha(t) = (r,t/\cosh(r))\]

For the case $v = 1$ the verification may be done in the upper half-plane model (i.e. $\mathbb{H} = \{x = (x,y) \in \mathbb{R}^2: y > 0\}$ with the metric $(\mathrm{d}x^2 + \mathrm{d}y^2)/y^2$).  It amounts to verifying that:
\[\frac{D}{\mathrm{d}t}\alpha'(t) = (0,1)\]
where
\[\alpha(t) = (t,1)\]

Finally, consider the case $v > 1$.  We will use a system of coordinates $(x,r)$ described as follows: Fix a unit speed geodesic $\gamma$ in the hyperbolic plane, the point with coordinates $(x,r)$ is the unique point at distance $r$ from $\gamma(x)$ on the geodesic ray with initial condition $i\gamma'(x)$.  In these coordinates the hyperbolic metric is $\cosh(r)^2\mathrm{d}x^2 + \mathrm{d}r^2$ as can be seen by the isometry with the upper half-plane model given by $\phi(x,r) = (e^x\tanh(r), e^x/\cosh(r))$.  One must verify that for each $r > 0$ the following curve is a solution to equation \ref{magneticequation}:
\[\alpha(t) = (t/\sinh(r),r)\]

In this case calculation yields:
\[\frac{D}{\mathrm{d}t}\alpha'(t) = (0,-1/\tanh(r))\]

The norm of $\alpha'(t)$ is $1/\tanh(r) = v$.  Since $\alpha$ is at a constant distance from $t \mapsto (t/\sinh(r),0)$ which is a geodesic it follows that the rate of escape of either curve is $R = \frac{1}{\sinh(r)}$.  By application of the identity \begin{footnote}{This identity can be obtained starting with $\tanh(\arctanh(x)) = x$ by squaring both sides and using the fact that $\cosh(y)^2 = 1 + \sinh(y)^2$.}\end{footnote}:
\[\sinh(\arctanh(x)) = \frac{x}{\sqrt{1 - x^2}}\]
it follows that:
\[R = \frac{1}{\sinh(r)} = \frac{1}{\sinh(\arctanh(1/v))} = \sqrt{v^2 - 1}\]

\addcontentsline{toc}{section}{Acknowledgments}
\section*{Acknowledgments}

The author would like to thank Gonzalo Contreras, Fran\c{c}ois Ledrappier, Rafael Potrie, Mart\'in Sambarino, Juliana Xavier, the participants and organizers of the ``Seminario Atl\'antico de Geometr\'ia 2010'', and the two journal referee's and board member, for contributing significantly to improving the quality and clarity of this work.

\appendix
\addcontentsline{toc}{section}{Appendix: Visibility Manifolds}
\section*{Appendix: Visibility Manifolds}
The purpose of this appendix is to present some facts about Busemann functions and Visibility Manifolds.

\begin{definition}[Asymptotic tangent vectors]\label{asymptotictangentvectors}
Let $H$ be a Hadamard manifold.  Two vectors $v,w \in S(H)$ are said to be asymptotic if:
\[d(\exp(tv),\exp(tw)) = O(1)\text{ when }t \to +\infty\]
\end{definition}

We observe that if $v,w \in S(p)$ for some $p \in H$ then $v$ and $w$ are asymptotic if and only if $v = w$.  This follows because $t \mapsto d(\exp(tv),\exp(tw))$ is a non-constant convex function (lemma \ref{convexity}) taking the value $0$ at $t = 0$ and therefore grows at least linearly.  Also, by the same argument, the condition that the distance in the definition be bounded is equivalent to it being $o(t)$ when $t \to +\infty$.  The main non-trivial fact we shall use about this equivalence relationship is the following:

\begin{lemma}\label{asymptoticquotienttopology}
Let $H$ be a Hadamard manifold.  If $\{v_n\}_{n \ge 1}, \{w_n\}_{n \ge 1} \subset S(H)$ and $v,w \in S(H)$ are such that:
\[v_n \to v\text{ when } n \to +\infty\]
\[w_n \to w\text{ when } n \to +\infty\]
\[v_n\text{ is asymptotic to }w_n\text{ for all }n\]
then $v$ is asymptotic to $w$.
\end{lemma}
\begin{proof}
Let $\pi: S(H) \to H$ denote the bundle projection and $g:\mathbb{R}\times S(H) \to S(H)$ the geodesic flow of $H$.

Suppose that there exists $t > 0$ such that:
\[d(\pi(g^tv),\pi(g^tw)) > d(\pi(v),\pi(w))\]

Then because $g$ is continuous there exists $n$ such that:
\[d(\pi(g^tv_n),\pi(g^tw_n)) > d(\pi(v_n),\pi(w_n))\]

However by convexity this implies that:
\[\lim_{s \to +\infty}d(\pi(g^sv_n),\pi(g^sw_n)) = +\infty\]

Which contradicts the fact that $v_n$ and $w_n$ are asymptotic.
\end{proof}

We recall the following definition of Busemann functions (compare with the definition of $f_\gamma$ in \cite[pg. 56]{MR0336648}).  See \cite{2010arXiv1010.6028F} for background on this concept.
\begin{definition}[Busemann function, horosphere]
Let $H$ be a Hadamard manifold and $v \in S(H)$.  The Busemann function $B_v: H\times H \to \mathbb{R}$ is defined by:
\[B_v(x,y) = \lim_{t \to +\infty}d(x,\exp(tv)) - d(\exp(tv),y)\]

The maximal sets of points $L \subset H$ with $B_v(x,y) = 0$ for all $x,y \in L$ are called horospheres or limit spheres of $H$.
\end{definition}

Directly from the definition one can deduce that both the inequality $|B_v(p,q)| \le d(p,q)$ and the equality $B_v(p,r) = B_v(p,q)+B_v(q,r)$ hold for all $p,q,r \in H$ and all $v \in S(H)$.  We shall also use two more facts.

\begin{lemma}\label{asymptoticbusemannfunctions}
Let $H$ be a Hadamard manifold.  If $v,w \in S(H)$ are asymptotic to each other then $B_v = B_w$.
\end{lemma}
\begin{proof}
Let $v \in S(p)$ and $w \in S(q)$. If one defines $f,g:\mathbb{H} \to \mathbb{R}$ by:
\[f(x) = B_v(p,x)\]
\[g(x) = B_w(q,x)\]

Then by \cite[Proposition 3.1]{MR0336648} one has that $f-g$ is constant.  Since:
\[B_v(x,y) - B_w(x,y) = f(y)-f(x) - (g(y) -g(x))\]
the conclusion follows.
\end{proof}

\begin{lemma}\label{convergenceofbusemannfunctions}
Let $H$ be a Hadamard manifold.  If $\{v_n\}_{n \ge 1} \subset S(H)$ and $v \in S(H)$ are such that:
\[v_n \to v\text{ when }n \to +\infty\]
then $B_{v_n} \to B_v$ uniformly on compact subsets of $H\times H$ when $n \to +\infty$.
\end{lemma}
\begin{proof}
The function $C:S(H)\times H \to \mathbb{R}$ defined by:
\[C(w,x) = \lim_{t\to +\infty}d(\exp(tw),x) - t\]
is continuous by \cite[Proposition 2.3]{MR0314084}.  We will show that this is equivalent to our claim.

For this purpose observe that if $w \in S(H)$ and $x,y \in H$ then:
\begin{align*}
B_w(x,y) &= \lim_{t \to+\infty}d(\exp(tw),x) - t + t - d(\exp(tw),y)
\\&= C(w,x)-C(w,y)
\end{align*}

It now follows from the the continuity of $C$ that $B_{v_n}$ converges to $B_v$ pointwise.  Uniform convergence on compact subsets is a consequence of the inequality $|B_w(x,y)| \le d(x,y)$ which is valid for all $w \in S(H)$ and $x,y \in H$.
\end{proof}

We will now restrict our attention to manifolds satisfying so called `visibility axioms' 1 and 2.  In these manifolds we will find a natural projection from the set $\mathcal{G} \subset S(H) \times S(H)$ of pairs of distinct unit tangent vectors with the same base point, to $S(H)$.

Our definition of `strong visibility' is equivalent to visibility along with Axiom 2 of \cite[Definition 4.1]{MR0336648} (here we use the fact that $H(\infty)$ with the cone topology is homeomorphic to $S(p)$ for all $p \in H$ as shown in \cite[Theorem 2.10]{MR0336648}).

\begin{definition}[Visibility manifold]
A visibility manifold is a Riemannian manifold whose universal Riemannian covering is a Hadamard manifold $H$ with the following property:  For each $p \in H$ and each $\epsilon > 0$ there exists a number $R = R(p,\epsilon)$ such that for any geodesic segment $[q,r] \subset H$ with $d(p,[q,r]) \ge R$ it holds that $\angle_p(q,r) \le \epsilon$.
\end{definition}

We will use the following result:
\begin{lemma}[Corollary 4.6 of \cite{MR0336648}]
\label{neighborhoodvisibilitylemma}
Let $H$ be a visibility Hadamard manifold.  For each compact $K \subset H$ and $\epsilon > 0$ there exists $R > 0$ such that any geodesic segment $[q,r] \subset H$ with $d(K,[q,r]) \ge R$ satisfies $\angle_p(q,r) \le \epsilon$ for all $p \in K$.
\end{lemma}

The existence part of the following definition was shown to hold on any visibility Hadamard manifold in \cite[Proposition 4.4]{MR0336648} and the converse implication was stated without proof.

\begin{definition}[Strong visibility manifold]
A strong visibility manifold is a visibility manifold such that its universal Riemannian covering space $H$ satisfies the following:  For each pair $(v^+,v^-) \in \mathcal{G}$ there exists a unique (up to reparametrization by time translation) unit speed geodesic $\alpha: \mathbb{R} \to H$ such that $\alpha'(0)$ is asymptotic to $v^+$ and $-\alpha'(0)$ is asymptotic to $v^-$.
\end{definition}

Note that it was shown by T.Ukai that any Hadamard manifold whose geodesic flow is of Anosov type is a strong visibility manifold (see \cite{MR1762792}).  To conclude we show that $\alpha$ depends continuously on $(v^+,v^-)$ in the above definitions.  This has been used to construct measurable semi-conjugating maps in section \ref{sectionbundles}.

\begin{lemma}[Projection along horospheres]\label{projectionalonghorospheres}
Let $H$ be a strong visibility Hadamard manifold.  There is a unique function $\pi_h: \mathcal{G} \to S(H)$ such that for all $(v^+,v^-) \in \mathcal{G}$ the following conditions are satisfied:
\begin{enumerate}
\item Setting $v = \pi_h(v^+,v^-)$ it holds that $v$ is asymptotic to $v^+$ and $-v$ is asymptotic to $v^-$.
\item If $p,q \in H$ are such that $v^+,v^- \in S(p)$ and $v \in S(q)$ then $B_v(p,q) = 0$.
\end{enumerate}
Furthermore, $\pi_h$ is continuous.
\end{lemma}
\begin{proof}
It follows from Lemma \ref{asymptoticbusemannfunctions} that if $v,w \in S(H)$ are asymptotic then $B_v = B_w$.  Also, because geodesics are globally minimizing, one has $B_v(\exp(sv),\exp(tv)) = t-s$.  Hence, given $(v^+,v^-) \in \mathcal{G}$ with basepoint $p$, if $\alpha$ is the unique unit speed geodesic (up to reparametrization by time translation) with $\alpha'(0)$ asymptotic to $v^+$ and $-\alpha'(0)$ asymptotic to $v^-$ then there is a unique $t \in \mathbb{R}$ with $B_{v^+}(\alpha(t),p) = 0$.  This shows that the function $\pi_h$ exists and is unique (in fact, in the above notation, $\pi_h(v^+,v^-) = \alpha'(t)$).

To establish continuity let $(v^+,v^-) \in \mathcal{G} \cap S(p)$ with $\pi_h(v^+,v^-) = v$ and let $q \in H$ be such that $v \in S(q)$.  Consider sequences $p_n$ and $v^+_n,v^-_n \in S(p_n)$ such that:
\[(v_n^+,v_n^-) \to (v^+,v^-)\text{ when }n \to +\infty\]
Finally, define $v_n = \pi_h(v_n^+,v_n^-)$ and let $q_n$ be such that $v_n \in S(q_n)$.  We must show that $v_n \to v$ when $n \to +\infty$.

First suppose $v_n \to w \in S(r)$ when $n \to +\infty$.  One has that $B_{v_n}(p_n,q_n) = 0$ for all $n$ and by Lemma \ref{convergenceofbusemannfunctions} this implies that $B_w(p,r) = 0$.  Also by Lemma \ref{asymptoticquotienttopology} one has that $w$ is asymptotic to $v^+$ and $-w$ to $v^-$.  By the uniqueness of $\pi_h$, which we have already established, it follows that $r = q$ and $w = v$.

The above argument implies that any convergent subsequence of $v_n$ must converge to $v$, hence to establish continuity it suffices to show that $\{v_n\}$ is bounded.

For this purpose take $0 < \epsilon < \inf_n\angle_{p_n}(\exp(v_n^+),\exp(v_n^-))$ and $K = \{p\} \cup \bigcup_n \{p_n\}$.  By Lemma \ref{neighborhoodvisibilitylemma} there exists a number $R$ and a sequence $\{t_n\} \subset \mathbb{R}$ such that:
\[d(\exp(t_nv_n),p_n) \le R\]
and therefore $|B_{v_n^+}(p_n,\exp(t_nv_n))| \le R$.

However since 
\[0 = B_{v_n}(p_n,q_n) = B_{v_n}(p_n,\exp(t_nv_n)) + B_{v_n}(\exp(t_nv_n), q_n) = O(1) - t_n\]
one obtains that $t_n = O(1)$ when $n \to +\infty$ and hence $\{v_n\}$ is bounded.
\end{proof}

\def\cprime{$'$}

\end{document}